\documentclass[11pt,letterpaper]{article}

\pdfoutput=1

\usepackage{graphics,epsfig,graphicx,color}
\usepackage[caption = false]{subfig}
\usepackage{float}
\usepackage{capt-of}

\usepackage{amsmath}
\numberwithin{equation}{section}

\usepackage{amssymb,latexsym}

\usepackage{setspace}

\usepackage{amsmath}

\usepackage{mathrsfs,amsfonts}

\usepackage{amsthm,amsxtra}

\usepackage[final]{showkeys}

\usepackage{stmaryrd}

\usepackage{algorithm}
\usepackage{algorithmicx}
\usepackage{algpseudocode}

\usepackage[openbib]{currvita}
\usepackage[square,numbers,sort]{natbib}

\usepackage{etaremune}
\usepackage{enumitem} 

\newif\ifPDF
\ifx\pdfoutput\undefined
	\PDFfalse
\else
	\ifnum\pdfoutput > 0
		\PDFtrue
	\else
		\PDFfalse
	\fi
\fi

\ifPDF
	\usepackage{pdftricks}
	\begin{psinputs}
		\usepackage{pstricks}
		\usepackage{pstcol}
		\usepackage{pst-plot}
		\usepackage{pst-tree}
		\usepackage{pst-eps}
		\usepackage{multido}
		\usepackage{pst-node}
		\usepackage{pst-eps}
	\end{psinputs}
\else
	\usepackage{pstricks}
\fi

\ifPDF
	\usepackage[debug,pdftex,colorlinks=true, 
	linkcolor=blue, bookmarksopen=false,
	plainpages=false,pdfpagelabels]{hyperref}
\else
	\usepackage[dvips]{hyperref}
\fi

\usepackage{fancyhdr}

\usepackage{comment}

\usepackage[toc,page]{appendix}

\usepackage{cancel}

\usepackage{multirow}
\usepackage{tabularx}

\newtheorem{theorem}{Theorem}[section]
\newtheorem{lemma}[theorem]{Lemma}
\newtheorem{definition}[theorem]{Definition}

\newtheorem{remark}[theorem]{Remark}

\newcommand{\bbR}{\mathbb R}

 \newcommand{\cH}{\mathcal H}
\newcommand{\cI}{\mathcal I} 
\newcommand{\cK}{\mathcal K} \newcommand{\cL}{\mathcal L}
 
\newcommand{\cO}{\mathcal O}

 \newcommand{\cV}{\mathcal V}

\newenvironment{keywords}
{\noindent{\bf Key words.}\small}{\par\vspace{1ex}}
\newenvironment{AMS}
{\noindent{\bf AMS subject classifications 2020.}\small}{\par}

\newcommand{\DELETE}[1]{}

\title{A discontinuous Galerkin method for one-dimensional nonlocal wave problems}

\author{Qiang Du \thanks{Department of Applied Physics and Applied Mathematics, Columbia University, New York, NY 10027, USA. Email: qd2125@columbia.edu} \quad\quad 
Kui Ren \thanks{Department of Applied Physics and Applied Mathematics, Columbia University, New York, NY 10027, USA.  Email: kr2002@columbia.edu} \quad\quad  Lu Zhang\thanks{Department of Computational Applied Mathematics and Operations Research, and Ken Kennedy Institute, Rice University, Houston, TX, 77005, USA.  Email: lz82@rice.edu}  \quad\quad Yin Zhou \thanks{Department of Applied Physics and Applied Mathematics, Columbia University, New York, NY 10027, USA. Email: yz3888@columbia.edu}
} 

\begin{document}

\maketitle

\begin{abstract}
    This paper presents a fully discrete numerical scheme for one-dimensional nonlocal wave equations and provides a rigorous theoretical analysis. To facilitate the spatial discretization, we introduce an auxiliary variable analogous to the gradient field in local discontinuous Galerkin (DG) methods for classical partial differential equations (PDEs) and reformulate the equation into a system of equations. The proposed scheme then uses a DG method for spatial discretization and the Crank-Nicolson method for time integration. We prove optimal $L^2$ error convergence for both the solution and the auxiliary variable under a special class of radial kernels at the semi-discrete level. In addition, for general kernels, we demonstrate the asymptotic compatibility of the scheme, ensuring that it recovers the classical DG approximation of the local wave equation in the zero-horizon limit. Furthermore, we prove that the fully discrete scheme preserves the energy of the nonlocal wave equation. A series of numerical experiments are presented to validate the theoretical findings.
\end{abstract}
 

\begin{keywords}
discontinuous Galerkin method, nonlocal wave equation, error estimates, numerical stability, asymptotic compatibility
\end{keywords}


\begin{AMS}
	45A05, 65M12, 65M60, 65R20
\end{AMS}

\section{Introduction}

In recent decades, nonlocal models have received considerable attention for their ability to capture long-range interactions, making them valuable tools in various scientific and engineering disciplines (see, e.g., \cite{DuHaZhZh-SIAM18,DuZh-JMMS23,Aksoylu-JPNM23,ZhHuDuZh-SIAM17,HeShSeCySi-IJNME23,WaYaZh-arXiv22,EmWe-MMS07,EmWe-CMS07,ChGo-ESAIM18,SachsSchu2013,AksoyluMengesha2010,DayalBhattacharya2007}). Unlike classical partial differential equations (PDEs), nonlocal models incorporate integral operators that account for interactions over a finite spatial neighborhood, allowing them to handle singularities and discontinuities more naturally. A notable application of nonlocal modeling arises in peridynamics, originally introduced by Silling \cite{Silling2000} as a reformulation of classical elastodynamics to model fracture and material damage. In this work, our study is related to a special case of the time-dependent linear peridynamic model of the following form
\begin{equation}\label{EQ:NVW}
\begin{aligned}
   u_{tt} + \cL u &= f \quad \ \ \text{on } \Omega, t>0, \\
   u({\bf x},0) &= u_0 \quad  \ \text{on } \Omega \cup \Omega_\cI, \\
   \cV u &= g \quad\quad  \text{on } \Omega_\cI, t>0,
\end{aligned}
\end{equation}
where $\Omega \subseteq \mathbb{R}^n$ denotes a bounded, open domain; $\cV$ is a linear operator of constraints acting on a volume $\Omega_\cI$ that is disjoint from $\Omega$; and the action of the linear operator $\mathcal{L}$ on the function $u(\mathbf{x})$ is defined as
\begin{equation}\label{eq:nonlocal_operator}
\mathcal{L} u(\mathbf{x}):=2 \int_{\Omega}(u(\mathbf{y})-u(\mathbf{x})) \gamma(\mathbf{x}, \mathbf{y}) d \mathbf{y} \quad \forall \mathbf{x} \in \Omega \subseteq \mathbb{R}^n,
\end{equation}
where $\gamma(\mathbf{x}, \mathbf{y}): \Omega \times \Omega \rightarrow \mathbb{R}$ is a symmetric, non-negative kernel function satisfying $\gamma(\mathbf{x}, \mathbf{y})=\gamma(\mathbf{y}, \mathbf{x}) \geq 0$. The operator $\cL$ is classified as non-local because its evaluation at any point $\bf x$ depends on the values of $u$ at all other points $\mathbf{y} \neq \mathbf{x}$ within the domain. The nonlocality allows $\cL$ to capture long-range interactions, making it fundamentally different from classical differential operators, which rely only on local information at the point $\bf x$. Nonlocal operators of the form \eqref{eq:nonlocal_operator}, including various generalizations of $\cL$, play a crucial role in many scientific and engineering applications. They have been widely used across various domains, including in image analysis for edge detection and denoising (\cite{buades2010image, gilboa2009nonlocal}), in machine learning for graph-based learning and kernel methods (\cite{rosasco2010learning, vishwanathan2010graph}), in phase transition models to describe nonlocal interactions in complex systems (\cite{bates1999integrodifferential, fife2003some}), in nonlocal heat conduction to capture anomalous diffusion phenomena (\cite{bobaru2010peridynamic, wangperi}), and in wave propagation models where long-range interactions and dispersion effects are essential (\cite{Silling2000, weckner2005effect}). These applications underscore the importance of nonlocal formulations in extending classical models to account for long-range dependencies and heterogeneities that are otherwise difficult to capture with purely local differential equations.

To date, extensive research has been done on nonlocal equations involving operators of the form \eqref{eq:nonlocal_operator}. On the theoretical side, various studies have investigated well-posedness, regularity, and convergence properties. For example, \cite{foss2016differentiability} studied a Dirichlet-type (volume-constrained) problem for scalar nonlocal equations, establishing a nonlocal counterpart to classical regularity theorems for elliptic systems. \cite{alali2021fourier} developed a Fourier multiplier framework for analyzing the periodic nonlocal Poisson equation and derived corresponding regularity results. In the context of nonlocal wave equations, \cite{beyer2016class} studied the problem in $\mathbb{R}^n$, proving that the governing operator is a bounded function of its classical counterpart and demonstrating strong resolvent convergence to classical solutions. \cite{dang2024regularity} further investigated the regularity properties of solutions on a periodic torus in $\mathbb{R}^n$. On the numerical side, various discretization techniques and computational methods have been developed in particular for nonlocal wave equations. For instance, \cite{guan2015stability} proposed a fully discretized scheme combining the Newmark method for time integration and piecewise linear finite element methods for spatial discretization in a volume-constrained setting, and analyzed its stability and convergence properties. \cite{coclite2020numerical} investigated high-order spatial discretization techniques for peridynamic wave models in $\mathbb{R}$, using composite quadrature formulas and spectral methods. \cite{alali2020fourier} developed spectral techniques for solving periodic nonlocal time-dependent problems in arbitrary dimensions, using Fourier multipliers of nonlocal Laplace operators to improve accuracy and computational efficiency.  \cite{du2018numerical} further proposed an absorbing boundary condition for nonlocal wave equations in unbounded domains using an integral equation method. For interested readers, we refer to \cite{du2012analysis,d2020numerical} and the references therein for a more comprehensive overview of theoretical developments and numerical methods of nonlocal equations involving operators of the form \eqref{eq:nonlocal_operator}.

The aim of this work is to develop numerical schemes for a nonlocal wave equation belonging to the category of \eqref{EQ:NVW}, together with a rigorous stability and convergence analysis. A key feature of \eqref{EQ:NVW} is its energy conservation property when the volume constraint $\cV$ satisfies certain conditions, as detailed in Appendix B of \cite{du2012analysis}. Therefore, it is crucial that the proposed numerical scheme preserves this property at the discrete level. Another important aspect of  many nonlocal models is their asymptotic behavior in the local limit. For example, as documented in \cite{du2012analysis}, as the support of the kernel function $\gamma$ shrinks to zero, the nonlocal operator $\cL$ approaches the classical local diffusive operator. In turn, the time-dependent nonlocal volume-constrained wave problem \eqref{EQ:NVW} converges to a classical wave equation. This property is crucial for the design of robust numerical schemes. A numerical scheme that correctly captures this limit behavior is called an \textit{asymptotically compatible} (AC) scheme, a concept introduced in \cite{tian2013analysis}. An AC scheme ensures that numerical solutions of the nonlocal model converge to the correct local solution as both the mesh size and the nonlocal interaction domain tend to zero. For further details on AC schemes, the reader is referred to \cite{du2019asymptotically,tian2017conservative}. Furthermore, it is known that the exact solution of the nonlocal wave equation \eqref{EQ:NVW} can exhibit spatial discontinuities for certain kernel functions $\gamma$. This motivates us to use discontinuous Galerkin (DG) methods for the spatial discretization. DG methods are a class of finite element methods that use discontinuous piecewise polynomials for both numerical solutions and test functions. Originally introduced by Reed and Hill in 1973 for neutron transport problems (\cite{ReHi-LASL73}), DG methods have since evolved into many variants, including local DG (\cite{ChouShuXing2014}), interior penalty DG (\cite{GroteSchneebeliSchotzau2006}),  hybridizable DG (\cite{CockburnGopalakrishnanLazarov2009}), energy DG (\cite{RenLuZhou2024}), and more. Their key advantages, such as arbitrary high-order accuracy, element-wise conservation, geometric flexibility, and $hp$-adaptivity, have made them widely applicable in numerous scientific and engineering applications, such as computational fluid dynamics, acoustics, and magnetohydrodynamics (see \cite{cockburn2004discontinuous,nguyen2011high} and references therein). Although many DG methods have been developed for local problems; however, their direct extension to nonlocal problems is not straightforward. This is because nonlocal models lack spatial differential operators, which precludes the use of integration by parts, a fundamental tool in DG formulations. In \cite{tian2015nonconforming}, the authors proposed a nonconforming DG method for nonlocal variational problems; however, the method lacks asymptotic compatibility. For nonlocal diffusion problems, asymptotically compatible penalty DG methods were introduced in \cite{DuJuLuTian-CAMC20, DuJuLuTian-ESAIM24}, where penalty terms involving jumps are constructed. Additionally, a DG method that closely follows the structure of local DG schemes while ensuring asymptotic compatibility was developed in \cite{DuJuLu-MC19}.

Building on these advances, we propose a stable, high-order, and asymptotically compatible DG method for solving the one-dimensional, time-dependent nonlocal wave equation \eqref{EQ:NW}, without introducing any penalty terms. To the best of our knowledge, this is the first study of asymptotic compatible DG methods applied to solve nonlocal wave equations  of the form \eqref{EQ:NVW}. As in 
 \cite{DuJuLu-MC19}, our approach begins with introducing an auxiliary variable to approximate the nonlocal counterpart of the gradient field, analogous to the local DG method for classical PDEs. Then a DG method is used to do the spatial discretization. We establish the stability and convergence of the  semi-discrete DG scheme for the nonlocal model. Moreover, we show that as the support of the kernel function vanishes, the scheme recovers the corresponding local DG scheme of the classical PDE limit for a fixed mesh size. For the time discretization in the numerical implementation, we use the Crank-Nicolson method, resulting in a fully discrete scheme that is implicit in time, unconditionally stable, and preserves the energy associated with the problem at the fully discrete level.

The rest of the paper is organized as follows. In Section \ref{sec:DG_stable}, we introduce the governing equations and formulate their DG discretization. We then analyze the stability of the proposed scheme, as stated in Theorem \ref{Th: stable}. Section \ref{sec:error} introduces projection operators and establishes optimal $L^2$ error estimates for the semi-discrete scheme when the kernel $\gamma$ is radial and belongs to $L_{\text{loc}}^1(\mathbb{R})$ (see Theorems \ref{theorem_ut} and \ref{theorem_u}). In addition, we further demonstrate the asymptotic compatibility of the scheme as presented in Theorem \ref{theorem_asy} for a general kernel $\gamma$. Section \ref{sec:time} extends the analysis to the fully discrete DG method coupled with the Crank-Nicolson time discretization and proves its energy-conserving properties in Theorem \ref{dis_thm}. Numerical experiments validating the theoretical results are presented in Section \ref{sec:numerical}. Finally, we conclude with a summary of our findings in Section \ref{sec:conclusion}.

\section{Semi-discrete DG formulation and its stability}\label{sec:DG_stable}
We consider the following one-dimensional, time-dependent, nonlocal wave equation:
\begin{equation}\label{EQ:NW}
   u_{tt}(x,t) + \cL_{\delta} u(x,t) = 0, \quad {(x,t)} \in (a,b) \times (0,T), 
\end{equation}
subject to periodic boundary and initial conditions:
\[u({x}, 0) = u_0({x}), \quad u_t({x}, 0) = u_1({x}), \ \quad { x} \in (a,b).\]
Here, $u_t$ and $u_{tt}$ are the first- and second-order time derivatives of $u(x,t)$, respectively. The nonlocal operator $\cL_\delta$ is defined as follows
\[\cL_{\delta} u(x,t) := -2 \int^{x+\delta}_{x-\delta} \big(u(y,t)-u(x,t)\big) \gamma(x,y) dy,\]
where $\delta$ is the horizon of the nonlocal operator specifying the domain of spatial interaction, and $\gamma(x,y)$ is a nonnegative, symmetric kernel function quantifying the influence of $u(y,t)$ on $u(x,t)$ within the horizon $\delta$.

 In this paper, we focus on the case where the kernel function is radial, satisfying $s^2\gamma(s) \in L^1_{\text{loc}}(\bbR)$. Specifically, the kernel satisfies
\begin{equation}\label{eq:kernel}
\gamma(x,y) = \gamma(x-y) \quad \text{ and } \quad 0 < C_{\delta} : = \int_{-\delta}^{\delta} s^2 \gamma(s) ds < \infty.
\end{equation}
To establish a connection with the classical local wave equation, we impose the condition $\lim_{\delta \rightarrow 0} C_\delta = 1$. Under this scaling, the nonlocal problem \eqref{EQ:NW} asymptotically converges to the following local wave equation:
\begin{equation}\label{EQ:wave}
u_{tt}(x,t) - u_{xx}(x,t) = 0, \quad (x,t) \in (a,b)\times(0,T],
\end{equation}
where $u_{xx}$ is the second-order spatial derivative of $u(x,t)$. This transition emphasizes the consistency of the nonlocal model with the classical wave equation as $\delta \rightarrow 0$, thus strengthening the validity of the nonlocal formulation under appropriate scaling of the kernel function.

\subsection{Spatial semi-discrete DG formulation}
\label{sec:semi_discrete}

In this section we develop a spatial semidiscrete DG formulation for the problem \eqref{EQ:NW}. For the spatial discretization, we adopt the formulation proposed in \cite{DuJuLu-MC19}. Let $s\in(0,\delta]$ and rewrite \eqref{EQ:NW} as
\begin{equation}\label{eq:NW3}
\begin{aligned}
u_{tt} - 2 \int^{\delta}_{0} s^2 \gamma(s) \frac{1}{s} \big(q(x,t;s) - q(x-s,t;s)\big) ds &= 0,\\
q(x,t;s) - \frac{1}{s} \big(u(x+s,t) - u(x,t)\big) &= 0,
\end{aligned}
\end{equation}
where we used the fact that
\[\int^{x+\delta}_{x-\delta} \big(u(y,t)-u(x,t)\big) \gamma(x,y) dy = \int^{\delta}_{0} \big(u(x+s,t) - 2u(x,t) + u(x-s,t)\big) \gamma(s) ds.\]

 To proceed, we divide the interval $I = (a,b)$ into non-overlapping subintervals $I_j = \big(x_{j-\frac{1}{2}}, x_{j+\frac{1}{2}}\big)$, where $a = x_{\frac{1}{2}} < x_{\frac{3}{2}} < \cdots < x_{N+\frac{1}{2}} = b$. We assume that the partition is regular, i.e., there exists a positive constant $\nu$ such that $\nu \max_j h_j \leq \min_j h_j$, where $h_j$ denotes the length of the subinterval $I_j$. Let $h:=\max_j h_j$, and on each subinterval $I_j$ we approximate the pair $(u,q)$ by its discrete counterparts $(u_h,q_h)$, each belonging to the space of piecewise polynomials of degree $k$:
\[V_h^k:= \{v_h({x})\big| v_h({x})\in\mathcal{P}^k(I_j), {x}\in I_j\quad \forall I_j\subset I\}.\]
Following \cite{DuJuLu-MC19}, we also define the operators $\cH_j$ and $\cK_j$ on each subinterval $I_j$ as follows:
\begin{equation}\label{eq:hk}
\cH_{j}(v,w;s) := \int_{I_j} \frac{1}{s} \big(v(x+s) - v(x)\big) w(x) dx,\quad \cK_{j}(v,w;s) := \int_{I_j} \frac{1}{s} \big(v(x) - v(x-s)\big) w(x) dx.
\end{equation}
With these definitions, we can formulate the spatial semi-discrete DG scheme for the reformulated problem \eqref{eq:NW3} as follows: find $(u_h, q_h) \in V_h^k \times V_h^k$ such that for any $(v_h, w_h) \in V_h^k \times V_h^k$,
\begin{equation}\label{EQ:DG_scheme}
\begin{aligned}
\int_{I_j}(u_h)_{tt} v_h dx - 2 \int_0^\delta s^2 \gamma(s) \cK_j(q_h,v_h;s) ds &= 0,\\
\int_{I_j} q_h w_h dx - \cH_j(u_h,w_h;s) &= 0.
\end{aligned}
\end{equation}

\begin{remark}\label{remark:reformulation2}
   In addition to the reformulation \eqref{eq:NW3}, an alternative approach is to introduce an auxiliary variable with $q(x,t; s) := \frac{1}{s}\big(u(x,t) - u(x-s,t)\big)$, which reformulates \eqref{EQ:NW} into
\[u_{tt} - 2 \int^{\delta}_{0} s^2 \gamma(s) \frac{1}{s} \big(q(x+s,t;s) - q(x,t;s)\big) ds = 0.\]
The corresponding spatial semidiscrete DG scheme for this alternative formulation is: find $(u_h, q_h) \in V_h^k \times V_h^k$ such that for any $(v_h, w_h) \in V_h^k \times V_h^k$,\begin{equation}\label{eq:DG_scheme1}\begin{aligned}\int_{I_j}(u_h)_{tt} v_h dx - 2 \int_0^\delta s^2 \gamma(s) \cH_j(q_h,v_h;s) ds &= 0,\\\int_{I_j} q_h w_h dx - \cK_j(u_h,w_h;s) &= 0.\end{aligned}
\end{equation}
\end{remark}

\begin{remark}\label{remark:limit}
For any $\eta \in V_h^k$, let $\eta^-$ and $\eta^+$ represent the left and right limits of $\eta$ at $x_{j+\frac{1}{2}}$, respectively. When $h_j$ is fixed and $\delta \rightarrow 0^+$, the DG scheme \eqref{EQ:DG_scheme} and \eqref{eq:DG_scheme1} become:
\begin{equation}\label{eq:LDG}
\begin{aligned}
\int_{I_j} (u_h)_{tt} v_h dx - q_h (x,t;0^+) (v_h)_x dx - (\hat{q}_h)_{j+\frac{1}{2}} (v^{-}_h)_{j+\frac{1}{2}} + (\hat{q}_h)_{j-\frac{1}{2}} (v^{+}_h)_{j-\frac{1}{2}} &= 0, \\
\int_{I_j} q_h (x,t;0^+) w_h dx - u_h (w_h)_x dx + (\hat{u}_h)_{j+\frac{1}{2}} (w^{-}_h)_{j+\frac{1}{2}} - (\hat{u}_h)_{j-\frac{1}{2}} (w^{+}_h)_{j-\frac{1}{2}} &= 0,
\end{aligned}
\end{equation}
where the numerical fluxes $(\hat{q}_h)_{j+\frac{1}{2}}, (\hat{u}_h)_{j+\frac{1}{2}}$ are given by
\begin{equation*}
\left\{
\begin{aligned}
(\hat{q}_h)_{j+\frac{1}{2}} &:= q_h\big(x^{+}_{j+\frac{1}{2}},t;0^{+}\big)\\(\hat{u}_h)_{j+\frac{1}{2}} &:= u_h \big(x^{-}_{j+\frac{1}{2}},t\big)
\end{aligned} \right.\quad \text{  and  } \quad\left\{ \begin{aligned}
(\hat{q}_h)_{j+\frac{1}{2}} &:= q_h\big(x^{-}_{j+\frac{1}{2}},t;0^{+}\big) \\(\hat{u}_h)_{j+\frac{1}{2}} &:= u_h \big(x^{+}_{j+\frac{1}{2}},t\big)
\end{aligned}\right.
\end{equation*}
for the DG schemes \eqref{EQ:DG_scheme} and \eqref{eq:DG_scheme1}, respectively. These are local DG schemes for the classical wave equation \eqref{EQ:wave}, using different alternating fluxes. 
\end{remark}

 The remainder of this paper will focus on the reformulation \eqref{eq:NW3} and its corresponding DG formulation \eqref{EQ:DG_scheme}. The analysis of the alternative approach described in Remark \ref{remark:reformulation2} follows a similar methodology and will be omitted for brevity.

\subsection{Energy conservation}\label{sec:stability}

In this section, we establish the stability of the DG scheme \eqref{EQ:DG_scheme}. Before proceeding, we present a lemma that plays a crucial role in both stability and error analysis.

\begin{lemma}\label{lemma1} 
(\cite{DuJuLu-MC19}) Let the operators $\cH_j$ and $\cK_j$ be defined as in \eqref{eq:hk}, and denote by
\begin{equation}\label{eq:def_KH}
\mathcal{K}(w,v;s) := \sum_j \mathcal{K}_j(w,v;s),\quad \mathcal{H}(v,w;s) := \sum_j\mathcal{H}_j(v,w;s).
\end{equation}
Then, the following identities hold for any $v, w \in V_h^k, s > 0$:
\[\cH(v,w;s) + \cK(w,v;s) = 0, \quad \cH(v,v;s) = - \frac{1}{2} \int_a^b \frac{1}{s} \big(v(x+s) - v(x)\big)^2 dx,\]
and
\[\cK(w,w;s) = \frac{1}{2} \int_a^b \frac{1}{s} \big(w(x+s) - w(x)\big)^2 dx.
\]
\end{lemma}

 We are now ready to establish the stability of the scheme \eqref{EQ:DG_scheme}, as shown in the following theorem.

\begin{theorem}\label{Th: stable}
Let the kernel $\gamma(x,y)$ satisfy the condition \eqref{eq:kernel}, and let $(u_{h},q_h) \in V_{h}^k \times V_h^k$ be the DG solution computed by the scheme \eqref{EQ:DG_scheme}. Define the discrete energy $E_h(t)$ as
\[{E}_h(t) := \int_a^b  \frac{1}{2}\big((u_h(x,t))_t\big)^2 dx + \int_0^\delta s^2 \gamma(s) \int_a^b   q_h^2(x,t;s) dx ds.\]
Then, $E_h(t)$ is conserved, namely,
\begin{equation*}\label{eq:discrete_stable}
    E_h(t) = E_h(0) \quad \forall t > 0.
\end{equation*}
\end{theorem}

\begin{proof}
Taking the time derivative of the second equation in \eqref{EQ:DG_scheme}, then multiplying the resulting equation by $2s^2\gamma(s)$, and integrating over $s$ in the interval $(0,\delta)$, we get 
    \begin{equation}\label{eq:stable_aux}
       \int_0^\delta 2s^2 \gamma(s) \int_{I_j} (q_h)_t w_h dx - 2\cH_j \big((u_h)_t, w_h; s\big) ds = 0.
    \end{equation}
Next, by choosing $(v_h, w_h) = \big((u_h)_t, q_h\big)$ in \eqref{EQ:DG_scheme} and \eqref{eq:stable_aux}, and summing the resulting two equations over all subintervals, we have
    \begin{equation*}
    \frac{d}{dt}E_h(t) = 2 \int_0^\delta s^2 \gamma(s) \Big(\cK (q_h, (u_h)_t; s) + \cH\big( (u_h)_t, q_h; s\big)\Big) ds = 0,
    \end{equation*}
where the last equality follows from Lemma \ref{lemma1}, thus completing the proof.
\end{proof}

 In the following, we discuss the error analysis for the proposed DG scheme \eqref{EQ:DG_scheme}. Let $C$ be a universal positive constant, independent of the mesh size $h$, although its value may vary between different lines. We also introduce the standard notation for Sobolev spaces as follows: For any integer $s>0$, let $W^{s, p}(I_j)$ denote the standard Sobolev space on subinterval $I_j$, with norm $\|\cdot\|_{s, p, I_j}$ and seminorm $|\cdot|_{s, p, I_j}$. For $p=2$ we denote $W^{s, p}(I_j)$ as $H^s(I_j)$, with corresponding norms $\|\cdot\|_{s, p, I_j}=\|\cdot\|_{H^s(I_j)}$ and seminorms $|\cdot|_{s, p, I_j}=|\cdot|_{H^s(I_j)}$. For $p=0$ we use the standard $L^s$ space with the norm $\|\cdot\|_{s, p, I_j}=\|\cdot\|_{L^s(I_j)}$. We also define the global norms and seminorms as $|\cdot|_{H^s(a,b)} = \sum_j |\cdot|_{H^s(I_j)}$ and $\|\cdot\|_{L^s(a,b)} = \sum_j \|\cdot\|_{L^s(I_j)}$.

\section{Error estimates}\label{sec:error}

In this section, we derive optimal error estimates for the DG scheme in the $L^2$ norm, assuming that $\gamma(s)\in L_{\text{loc}}^1(\mathbb{R})$ and the horizon $\delta > 0$ being fixed. In addition, we prove suboptimal error estimates in the $L^2$ norm for a more general class of kernel functions. These results not only establish the convergence order of the scheme, but also show that the proposed DG scheme is asymptotically compatible (see the definition presented in Definition \ref{def:asym}). In Section \ref{projection}, we first review several projections and their corresponding properties, which are fundamental to our analysis.

\subsection{Projections}\label{projection}

For any $k\geq0$, we define the $L^2$-projection $P_h$ onto $V_h^k$ such that, for any $u\in H^{k+1}(a,b)$, the projection $P_h$ satisfies
\begin{equation}\label{eq:l2_projection}
\int_{I_j} (u - P_h u)v\ dx = 0\quad \quad \forall v\in \mathcal{P}^{k}(I_j), j = 1,\cdots,N.
\end{equation}
In addition, we define the Gauss-Radau projection $P_h^{\pm}$ onto $V_h^k$ as follows: For any $u\in H^{k+1}(a,b)$, the projection $P_h^\pm$ satisfies
\begin{equation}\label{eq:gr_projection}
\begin{aligned}
&\int_{I_j} (u - P_h^\pm u)v\ dx = 0\quad \quad \forall v\in \mathcal{P}^{k-1}(I_j), j = 1,\cdots,N,\\
& P_h^+u(x_{j-\frac{1}{2}}^+) = u(x_{j-\frac{1}{2}}^+) \quad \text{and} \quad P_h^-u(x_{j+\frac{1}{2}}^-) = u(x_{j+\frac{1}{2}}^-),
\end{aligned}
\end{equation}
for $k\geq 1$. When $k=0$, we only require the second line in \eqref{eq:gr_projection} to define the Gauss-Radau projection $P_h^{\pm}$.

Then, from \cite{ciarlet2002finite}, we have the following approximation properties hold: For sufficiently small $h$ there exists a positive constant $C$, dependent of $k$ but independent of the spatial mesh size $h$, such that for any $u\in H^{k+1}(a,b)$,
\begin{equation}\label{projrel}
\|u - P_hu\|_{L^2(a,b)} + \|u - P_hu\|_{L^\infty(a,b)} \leq Ch^{k+1},
\end{equation}
and
\begin{equation}\label{projre_gr}
\|u - P_h^\pm u\|_{L^2(a,b)} + \|u - P_h^\pm u\|_{L^\infty(a,b)} \leq Ch^{k+1}.
\end{equation}
Furthermore, for $v \in V_h^k$, we have the following inverse inequalities (\cite{ciarlet2002finite}):
\begin{equation}\label{eq:inverse}
\|v_x\|_{L^2(a,b)} \leq Ch^{-1}\|v\|_{L^2(a,b)} \quad \text{and} \quad \|v\|_{L^\infty(a,b)} \leq Ch^{-\frac{1}{2}}\|v\|_{L^2(a,b)}.
\end{equation}

\subsection{Optimal error estimate for \texorpdfstring{$\gamma(s) \in L^1_{\text{loc}} (\bbR)$}{} and fixed horizon \texorpdfstring{$\delta$}{}}

We first establish error estimates for the proposed DG scheme \eqref{EQ:DG_scheme} with a fixed horizon $\delta > 0$. To this end, let us denote the errors by
\[\begin{aligned}
\tilde e_u = e_u + \delta_u &:= \big(u_h - u\big) + \big(u - P_h u\big),\\
\tilde e_q = e_q + \delta_q &:= \big(q_h - q\big) + \big(q - P_h q\big),
\end{aligned}\]
where the $L^2$-projection, $P_h$, is defined in \eqref{eq:l2_projection}. Since both the continuous solution $(u,q)$ and the DG solution $(u_h,q_h)$ satisfy the DG scheme \eqref{EQ:DG_scheme}, the following error equations hold for any $(v_h,w_h) \in V_h^k \times V_h^k$:
\begin{equation}\label{eq:error_aux1}
\begin{aligned}
\int_{I_j} (e_u)_{tt} v_h dx - 2 \int_0^\delta s^2 \gamma(s) \mathcal{K}_j\big(e_q, v_h; s\big)ds &= 0,\\
\int_{I_j} e_q w_h dx - \mathcal{H}_j\big(e_u, w_h; s\big) &= 0,
\end{aligned}
\end{equation}
where the operators $\cH_j$ and $\cK_j$ are defined in \eqref{eq:hk}. To proceed, we first derive error estimates for $(u_t,q)$, as shown in the following theorem, and then establish error estimates for $u$.

\begin{theorem}\label{theorem_ut}
Let the horizon $\delta > 0$ be fixed. Assume that the kernel $\gamma(s)$ satisfies the condition \eqref{eq:kernel} and moreover $\gamma(s)\in L_{\text{loc}}^1(\mathbb{R})$. Suppose further that $(u_h, q_h) \in V_h^k \times V_h^k$ is the DG solution of the problem \eqref{eq:NW3}. If the exact solution $(u,q)$ of the problem \eqref{eq:NW3} is sufficiently smooth, then the following optimal error estimate holds:
    \begin{equation}\label{eq:theorem_ut}
    \begin{aligned}
        & \frac{1}{2} \|(e_u)_t\|_{L^2(a,b)}^2   + \int_0^\delta s^2 \gamma(s) \|e_q^2(x,t;s)\|_{L^2(a,b)}^2  ds \le C(\delta)  h^{2k+2}\quad \forall t\in[0,T],
    \end{aligned}
    \end{equation}
where $C(\delta)>0$ is a constant independent of the mesh size $h$, but dependent of the kernel $\gamma(s)$, the horizon $\delta$, and the norms $|u_t(x,t)|_{H^{k+1}(a,b)}$ and $|q(x,t; s)|_{H^{k+1}(a,b)}$.
\end{theorem}

\begin{proof}

Taking the first-order time derivative of the second equation in \eqref{eq:error_aux1}, multiplying the resulting equation by $2s^2\gamma(s)$, and integrating over $(0,\delta)$ with respect to $s$, then summing the resulting equations over all subcells $I_j$, we obtain
\begin{equation}\label{eq:error_aux2}
\begin{aligned}
&\int_a^b (\tilde e_u)_{tt} v_h dx + 2\int_0^\delta s^2 \gamma(s) \int_a^b   (\tilde e_q)_t w_h dx ds \\
=\ & 2 \int_0^\delta s^2 \gamma(s) \Big(\mathcal{K}\big(\tilde e_q - \delta_q, v_h; s\big) + \mathcal{H}\big((\tilde e_u)_t - (\delta_u)_t, w_h; s\big)  \Big)ds,
\end{aligned}
\end{equation}
where we used the fact that $\tilde e_u = e_u + \delta_u$, $\tilde e_v = e_v + \delta_v$, along with the $L^2$-projection property \eqref{eq:l2_projection}. Substituting $v_h$ by $(\tilde e_u)_t$ and $w_h$ by $\tilde e_q$ in \eqref{eq:error_aux2}, and applying Lemma \ref{lemma1}, we get
\begin{equation}\label{eq:error_aux3}
\begin{aligned}
& \frac{d}{dt}\bigg( \frac{1}{2}\|(\tilde e_u)_t\|_{L^2(a,b)}^2 + \int_0^\delta s^2 \gamma(s) \|\tilde e_q(x,t;s)\|_{L^2(a,b)}^2  ds\bigg) \\
= \ & -2 \int_0^\delta s^2 \gamma(s) \Big(\mathcal{K}\big(\delta_q, (\tilde e_u)_t; s\big) + \mathcal{H}\big((\delta_u)_t, \tilde e_q; s\big)  \Big)ds.
\end{aligned}
\end{equation}
To estimate the terms involving $\cK$ and $\cH$ in \eqref{eq:error_aux3}, we use the definitions of $\cK_j$ and $\cH_j$ in \eqref{eq:hk}, apply H\"older's inequality, and utilize the projection property \eqref{projrel} to collect
\begin{equation}\label{eq:error_aux4}
\begin{aligned}
&-2\int_0^\delta s^2 \gamma(s)\mathcal{K}\big(\delta_q, (\tilde e_u)_t; s\big) ds = -2\int_0^\delta \mathcal{K}\Big(s\sqrt{\gamma(s)}\delta_q, s\sqrt{\gamma(s)}(\tilde e_u)_t; s\Big) ds\\
\leq \ & \int_0^\delta \bigg( 2\gamma(s) \|\delta_q(x, t;s) - \delta_q(x-s, t; s)\|_{L^2(a,b)}^2+ \frac{1}{2}s^2 \gamma(s) \|(\tilde e_u)_t\|_{L^2(a,b)}^2\bigg) ds \\
\leq \ & C h^{2k+2}\int_0^\delta \gamma(s) ds +  \frac{1}{2} \|(\tilde e_u)_t\|_{L^2(a,b)}^2, 
\end{aligned}
\end{equation}
and
\begin{equation}\label{eq:error_aux5}
\begin{aligned}
& -2\int_0^\delta s^2 \gamma(s) \mathcal{H}\big((\delta_u)_t, \tilde e_q; s\big) ds  = -2\int_0^\delta \mathcal{H}\Big(s\sqrt{\gamma(s)}(\delta_u)_t, s\sqrt{\gamma(s)}\tilde e_q; s\Big) ds \\
\leq \ & \int_0^\delta \bigg( \gamma(s) \|(\delta_{u})_t(x+s, t;) - (\delta_u)_t(x, t)\|_{L^2(a,b)}^2+ s^2 \gamma(s) \|\tilde e_q(x,t;s)\|_{L^2(a,b)}^2\bigg) ds \\
\leq\  & C h^{2k+2}\int_0^\delta \gamma(s)ds +\int_0^\delta s^2 \gamma(s)  \|\tilde e_q(x,t;s)\|_{L^2(a,b)}^2 ds. 
\end{aligned}
\end{equation}
Substituting \eqref{eq:error_aux4}--\eqref{eq:error_aux5} into \eqref{eq:error_aux3} we arrive at
\begin{equation}\label{eq:error_aux6}
\begin{aligned}
&\frac{d}{dt}\bigg(  \frac{1}{2}\|(\tilde e_u)_t\|_{L^2(a,b)}^2 + \int_0^\delta s^2 \gamma(s) \|\tilde e_q(x,t;s)\|_{L^2(a,b)}^2 ds\bigg)  \\ \leq& \  C(\delta) h^{2k+2} + \frac{1}{2}\|(\tilde e_u)_t\|_{L^2(a,b)}^2 + \int_0^\delta s^2 \gamma(s) \|\tilde e_q(x,t;s)\|_{L^2(a,b)}^2 ds.
\end{aligned}
\end{equation}

\noindent Finally, applying the triangle inequality, the projection property \eqref{projrel}, and Gr\"onwall's inequality yields the desired result \eqref{eq:theorem_ut}.
\end{proof}

Using the result derived in Theorem \ref{theorem_ut}, we now establish the error estimate for $u$.

\begin{theorem}\label{theorem_u}
    Let the horizon $\delta > 0$ be fixed. Assume that the kernel $\gamma(s)$ satisfies the condition \eqref{eq:kernel} and moreover $\gamma(s)\in L_{\text{loc}}^1(\mathbb{R})$. Suppose further that $(u_h, q_h) \in V_h^k \times V_h^k$ is the DG solution of the problem \eqref{eq:NW3}. If the exact solution $(u,q)$ of the problem \eqref{eq:NW3} is sufficiently smooth, then the following optimal error estimate holds:
    \begin{equation}\label{eq:l2eu}
    \begin{aligned}
        &\| e_u(x,t) \|_{L^2(a,b)}^2 \le C(\delta) h^{2k+2} \quad \forall t\in[0,T],
    \end{aligned}
    \end{equation}
where $C(\delta)>0$ is a constant independent of the mesh size $h$, but dependent of $\gamma(s)$, the horizon $\delta$, and the norms $|u_t(x,t)|_{L^\infty([0,T];H^{k+1}(a,b))}$, $|u(x,t)|_{L^\infty([0,T];H^{k+1}(a,b))}$ and $|q(x,t; s)|_{L^\infty([0,T];H^{k+1}(a,b))}$.
\end{theorem}
\begin{proof}

For any $t\leq T$, we define the time-integrated errors as follows
\begin{equation*}
\cI_{e_u}(x,t) := \int_t^\tau e_u(x,r) dr, \quad \cI_{\tilde e_u}(x,t) = \int_t^\tau \tilde e_u(x,r) dr, \quad \cI_{\delta_u}(x,t) := \int_t^\tau \delta_u(x,r) dr,
\end{equation*}
and
\begin{equation*}
\cI_{e_q}(x,t; s) := \!\!\int_t^\tau \!\!\!\!e_q(x,r;s) dr, \quad \cI_{\tilde e_q}(x,t;s) = \!\!\int_t^\tau \!\!\tilde e_q(x,r;s) dr, \quad \cI_{\delta_q}(x,t) := \!\int_t^\tau \!\delta_q(x,r;s) dr.
\end{equation*}

Then, taking the integral over $(t,\tau)$ with respect to the time variable for the second equation in \eqref{eq:error_aux1}, choosing $w_h = \tilde e_q$, multiplying the resulting equation by $2s^2\gamma(s)$, and integrating over $(0,\delta)$ in $s$ using the relation $e_q = \tilde e_q - \delta_q$, we get
\begin{equation}\label{eq:aux8}
\int_0^\delta \!\!\!2s^2\gamma(s)\!\!\int_{I_j}\!\! \big(\cI_{\tilde e_q}(x,t;s) \!-\! \cI_{\delta_q}(x,t;s)\big)  \tilde e_q(x,t;s) dx ds -2\! \int_0^\delta \!\!\!\!s^2\gamma(s)\mathcal{H}_j\big(\cI_{\tilde e_u} - \cI_{\delta_u}, \tilde e_q; s\big) ds = 0.
\end{equation}
Similarly, choosing $v_h = \cI_{\tilde e_u}(x,t)$ in the first equation of \eqref{eq:error_aux1}, and using the relation $e_u = \tilde e_u - \delta_u$ leads to
\begin{equation}\label{eq:aux9}
\int_{I_j} (\tilde e_u - \delta_u)_{tt} \cI_{\tilde e_u}dx - 2 \int_0^\delta s^2 \gamma(s) \mathcal{K}_j\big(\tilde e_q - \delta_q, \cI_{\tilde e_u}; s\big)ds = 0.
\end{equation}
Adding \eqref{eq:aux8} and \eqref{eq:aux9} together and summing the resulting equation over all subcells $I_j$ and using Lemma \ref{lemma1}, we obtain
\begin{equation}\label{eq:aux_10}
\begin{aligned}
\frac{d}{dt} \bigg( \frac{1}{2}\|\tilde e_u\|_{L^2(a,b)}^2  - \int_0^\delta s^2 &\gamma(s) \|\cI_{\tilde e_q}(x,t;s)\|_{L^2(a,b)}^2 ds\bigg) 
=  -\int_a^b \frac{d}{dt} \big((\tilde e_u)_{t} \cI_{\tilde e_u}\big) dx  \\
&+ 2 \int_0^\delta s^2 \gamma(s) \Big(-\mathcal{K}\big({\delta_q}, \cI_{\tilde e_u}; s\big) + \mathcal{H}\big(\cI_{\delta_u}, \tilde e_q; s\big)  \Big)ds,
\end{aligned}
\end{equation}
where we used the definition of the $L^2$-projection in \eqref{eq:l2_projection}. Integrating \eqref{eq:aux_10} over $(0,\tau)$ with respect to $t$, we then arrive at
\begin{equation}\label{eq:aux6}
\begin{aligned}
&\frac{1}{2} \|\tilde e_u(x,\tau)\|_{L^2(a,b)}^2 - \frac{1}{2} \|\tilde e_u(x,0)\|_{L^2(a,b)}^2 + \int_0^\delta s^2 \gamma(s) \|\cI_{\tilde e_q}(x,0;s)\|_{L^2(a,b)}^2ds\\
= &\!\! \int_a^b\!\!\! \big(\tilde e_u(x,0)\big)_t \cI_{\tilde e_u}(x,0)dx + 2\!\int_0^\tau \!\!\!\int_0^\delta\!\! s^2 \gamma(s) \Big(\!-\mathcal{K}\big( {\delta_q}, \cI_{\tilde e_u}; s\big) + \mathcal{H}\big(\cI_{\delta_u}, {\tilde e_q}; s\big) \! \Big)ds dt.
\end{aligned}
\end{equation}
Let $(u_h(x,0), q_h(x, 0;s)) = (P_hu(x,0), P_hq(x,0;s))$ and invoke the bound on $\|(e_u)_t\|_{L^2(a,b)}^2$ from Theorem \ref{theorem_ut}, and use H\"older's inequality, we derive from \eqref{eq:aux6} that
\begin{equation}\label{eq:aux11}
\begin{aligned}
&\frac{1}{2} \|\tilde e_u(x,\tau)\|_{L^2(a,b)}^2 + \int_0^\delta s^2 \gamma(s) \|\cI_{\tilde e_q}(x,0;s)\|_{L^2(a,b)}^2ds\\
\leq & \ C(\delta)h^{2k+2} \!\!+\! \frac{1}{8T^2}\|\cI_{\tilde e_u}(x,0)\|_{L^2(a,b)}^2 \!+ 2\!\int_0^\tau \!\!\!\int_0^\delta\!\! \!\!s^2 \gamma(s) \Big(\!\!-\!\mathcal{K}\big( {\delta_q}, \cI_{\tilde e_u}; s\big) \!+\! \mathcal{H}\big(\cI_{\delta_u}, {\tilde e_q}; s\big) \! \Big)ds dt.
\end{aligned}
\end{equation}
Now, following similar steps as in the derivation of \eqref{eq:error_aux4} and \eqref{eq:error_aux5}, and using the properties of the $L^2$-projection in \eqref{projrel}, the property of the kernel in \eqref{eq:kernel}, and H\"older's inequality, we generate
\begin{equation}\label{eq:aux_12}
\begin{aligned}
&2\int_0^\tau\int_0^\delta s^2 \gamma(s) \Big(\!-\mathcal{K}\big({\delta_q}, \cI_{\tilde e_u}; s\big) + \mathcal{H}\big(\cI_{\delta_u}, {\tilde e_q}; s\big) \Big)ds dt \\
= & \ 2\!\!\int_0^\tau\!\!\!\int_0^\delta \!\!\!-\mathcal{K}\Big(\!s\sqrt{\gamma(s)}{\delta_q}, s\sqrt{\gamma(s)}\cI_{\tilde e_u}; s\Big) \!+\! \mathcal{H}\Big(\!s\sqrt{\gamma(s)}\cI_{\delta_u}, s\sqrt{\gamma(s)}{\tilde e_q}; s\Big)ds dt \\
\leq & \ \frac{1}{8T^2}  \max_{t\in[0,T]}\|\cI_{\tilde e_u}\|_{L^2(a,b)}^2 + C(\delta)h^{2k+2}. 
\end{aligned}
\end{equation}
Next, we proceed to bound $\|\cI_{\tilde e_u}\|_{L^2(a,b)}^2$ with $\|\tilde e_u\|_{L^2(a,b)}^2$. To this end, for any $t<\tau\leq T$, using Cauchy-Schwarz inequality, we have
\begin{equation}\label{eq:aux7}
\|\cI_{\tilde e_u}\|_{L^2(a,b)}^2 = \int_a^b \Big(\int_t^\tau \tilde e_u(x,r) dr \Big)^2 dx \leq  (\tau-t)\int_t^\tau \|\tilde e_u\|_{L^2(a,b)}^2 dr \leq T^2\max_{t\in[0,T]}\|\tilde e_u\|_{L^2(a,b)}^2.
\end{equation}
Plugging \eqref{eq:aux_12}--\eqref{eq:aux7} into \eqref{eq:aux11} then yields
\begin{equation}\label{eq:aux13}
\max_{t\in[0,T]}\|\tilde e_u\|_{L^2(a,b)}^2 \leq C(\delta)h^{2k+2}.
\end{equation}
Last, using the triangle inequality, the projection property \eqref{projrel}, and the fact that $\tilde e_u = e_u + \delta_u$, we arrive at the desired result \eqref{eq:l2eu}.

\end{proof}

\begin{remark}
We note that the dependence of the convergence constant $C(\delta)$ on both the horizon $\delta$ and the kernel $\gamma(s)$ in Theorem \ref{theorem_ut} and Theorem \ref{theorem_u} follows from the calculation of $\int_0^\delta \gamma(s)ds$, which is obtained using H\"older's inequality when estimating the $\cK$ term in \eqref{eq:error_aux4} and \eqref{eq:aux_12}, and the $\cH$ term in \eqref{eq:error_aux5} and \eqref{eq:aux_12}. 
\end{remark}

In the following, we further demonstrate that the proposed DG scheme \eqref{EQ:DG_scheme} is asymptotically compatible (see Definition \ref{def:asym}) by establishing a suboptimal error estimate whose convergence constant is independent of the horizon $\delta$, the kernel function $\gamma(s)$, and the spatial mesh size $h$.

\subsection{Asymptotic compatibility with general kernel}\label{sec:asy}

We first define the concept of asymptotic compatibility for the DG scheme \eqref{EQ:DG_scheme} in the context of the nonlocal wave problem \eqref{EQ:NW}, consistent with the definition given in \cite{TiDu-SIAM14}. 
\begin{definition}\label{def:asym}
Let $u(x,t;\delta)$ be the solution to the nonlocal wave equation \eqref{EQ:NW}, $u_{\text{loc}}(x,t)$ be the solution to the local problem \eqref{EQ:wave}, and $u_h(x,t;\delta)$ be the DG approximation to the nonlocal wave problem \eqref{EQ:NW}. The DG scheme \eqref{EQ:DG_scheme} is then said to be asymptotically compatible if 
\begin{equation*}
\lim_{\delta,h \rightarrow 0}\|u_h(x,t;\delta) - u_{\text{loc}}(x,t)\|_{L^2(a,b)} = 0\quad \forall t > 0.
\end{equation*}
\end{definition}
 
 Note that on the continuum level, the nonlocal wave problem \eqref{EQ:NW} asymptotically converges to the local wave equation \eqref{EQ:wave} as $\delta \rightarrow 0$, i.e.,
\begin{equation*}
\lim_{\delta \rightarrow 0} \|u(x,t;\delta) - u_{\text{loc}}(x,t)\|_{L^2(a,b)} = 0\quad \forall t > 0.
\end{equation*}
Using Definition \ref{def:asym} and the triangle inequality, we can deduce that 
\begin{equation}
\lim_{\delta,h \rightarrow 0}\|u(x,t;\delta) - u_h(x,t;\delta)\|_{L^2(a,b)} = 0 \quad \forall t > 0
\end{equation}
is sufficient to guarantee the asymptotic compatibility of the DG scheme \eqref{EQ:DG_scheme}. In this paper, following a similar procedure in \cite{DuJuLu-MC19}, we can derive the following error estimate to ensure the asymptotic compatibility of the proposed DG scheme \eqref{EQ:DG_scheme}.

\begin{theorem}\label{theorem_asy}
    Suppose that the kernel $\gamma(s)$ satisfies the condition \eqref{eq:kernel}. Let $(u_h, q_h) \in V_h^k\times V_h^k$ be the DG approximation of the exact solution $(u,q)$ of the reformulated problem \eqref{eq:NW3}. If the exact solution $(u,q)$ is sufficiently smooth, the following suboptimal error estimates hold for any $t\in[0,T]$
\begin{equation}\label{eq:asy_theorem}
    \begin{aligned}
        &\|e_u\|_{L^2(a,b)}^2  \leq Ch, \quad\quad  \int_0^\delta s^2 \gamma(s) \|e_q(x,t;s)\|_{L^2(a,b)}^2  ds \le Ch, \quad \text{ when } k = 0 \text{ and } h_j = h,\\
        & \|e_u\|_{L^2(a,b)}^2  \leq Ch^{2k} , \quad \int_0^\delta s^2 \gamma(s) \|e_q(x,t;s)\|_{L^2(a,b)}^2  ds \le Ch^{2k}, \quad \text{when } k \geq 1,
    \end{aligned}
    \end{equation}
where $C>0$ is a constant independent of the mesh size $h$, the kernel $\gamma(s)$, and the horizon $\delta$, but depends on the norms $|u_t(x,t)|_{L^\infty([0,T]; H^{k+1}(a,b))}$, $|u(x,t)|_{L^\infty([0,T];H^{k+1}(a,b))}$ and $|q(x,t; s)|_{L^\infty([0,T];H^{k+1}(a,b))}$.
\end{theorem}
\begin{proof}
The derivation follows a similar approach as presented in \cite{DuJuLu-MC19}. In our case, however, we must first obtain an estimate for $\|(e_u)_t\|_{L^2(a,b)}$ (\underline{\emph{step one}}), which is then used to derive an estimate for $\|e_u\|_{L^2(a,b)}$ (\underline{\emph{step two}}). Moreover, since we are only able to derive a suboptimal convergence when $k \geq 1$, we divide our discussion into two cases with $k\geq 1$ and $k=0$. In the following, we first consider the case where $k\geq 1$, and then with $k = 0$. In addition, for each subcase, we present two-step process to estimate $\|e_u\|_{L^2(a,b)}$.

 \textbf{Case I:} $k\geq 1$. 

\underline{\emph{Step one.}}  To derive an estimate for $\|e_u\|_{L^2(a,b)}$, we first derive an estimate for $\|(\tilde e_u)_t\|_{L^2(a,b)}$. To this end, following the error equation \eqref{eq:error_aux3}, we need to establish a bound for 
\begin{equation}\label{eq:asy_eq0}
-2\int_0^\delta s^2 \gamma(s) \Big(\mathcal{K}\big(\delta_q, (\tilde e_u)_t; s\big) + \mathcal{H}\big((\delta_u)_t, \tilde e_q; s\big)  \Big)ds.
\end{equation}
To achieve this, we consider two cases with $\delta < h_j$ and $\delta \geq h_j$. In particular, when $\delta < h_j$, invoking the definition of the operator $\cK$ in \eqref{eq:def_KH}, for the first term in \eqref{eq:asy_eq0}, we have 
\begin{equation}\label{eq:asy_aux2}
\begin{aligned}
& -2\int_0^\delta s^2 \gamma(s) \cK\big(\delta_q, (\tilde e_u)_t;s\big) ds \\
= & -2\sum_j\int_0^\delta s^2 \gamma(s) \int_{I_j} \frac{1}{s} \big(\delta_q(x+s,t) - \delta_q(x,t)\big) \big(\tilde e_u(x,t)\big)_t dx ds \\
= &\ 2\sum_j \!\!\int_0^\delta \!\!\! s^2\gamma(s) \!\bigg(\!\int_{x_{j-\frac{1}{2}}}^{x_{j-\frac{1}{2}}+s}\!\frac{1}{s} \delta_q(x,t)\big(\tilde e_u(x,t)\big)_t dx - \!\!\!\int_{x_{j+\frac{1}{2}}}^{x_{j+\frac{1}{2}}+s} \!\frac{1}{s} \delta_q(x,t)\big(\tilde e_u(x-s,t)\big)_t dx \!\!\bigg) ds \\
 & -  2\sum_j\int_0^\delta s^2\gamma(s)\int_{x_{j-\frac{1}{2}}+s}^{x_{j+\frac{1}{2}}} \frac{1}{s} \Big(-\big(\tilde e_u(x,t)\big)_t + \big(\tilde e_u(x-s,t)\big)_t\Big)\delta_q(x,t) dxds \\
\leq & \ 2\sum_j \int_0^\delta s^2\gamma(s) \Big(2\|\delta_q\|_{L^\infty(a,b)} \|(\tilde e_u)_t\|_{L^\infty(a,b)} + (h_j-s) \|\delta_q\|_{L^\infty(a,b)}\|(\tilde e_u)_{tx}\|_{L^\infty(a,b)}\Big)ds\\
\leq &\ Ch^{k} \|(\tilde e_u)_t\|_{L^2(a,b)} \sum_j\int_0^\delta s^2\gamma(s) ds,
\end{aligned}
\end{equation}
where the last inequality is derived based on the property of the $L^2$-projection \eqref{projrel}, and the inverse inequality \eqref{eq:inverse}. Furthermore, applying H\"older's inequality, \eqref{eq:asy_aux2} yields
\begin{equation}\label{eq:asy_1aux2}
-2\int_0^\delta s^2\gamma(s) \cK(\delta_q, (\tilde e_u)_t; s)ds \leq Ch^{2k} + \frac{1}{2}\|(\tilde e_u)_t\|_{L^2(a,b)}^2. 
\end{equation}

 When $\delta \geq h_j$, we get  
\begin{equation}\label{eq:asy_aux1}
-2\int_0^\delta s^2 \gamma(s) \cK\big(\delta_q, (\tilde e_u)_t;s\big) ds 
=  -2\sum_j \Big( \int_0^{h_j} + \int_{h_j}^\delta \Big) s^2 \gamma(s) \cK_j\big(\delta_q, (\tilde e_u)_t;s\big) ds.
\end{equation}
Then, similar to the derivation of \eqref{eq:asy_aux2}, for the first term on the right hand side of \eqref{eq:asy_aux1} we can derive that
\begin{equation}\label{eq:asy_aux2_1}
\begin{aligned}
 & -2\sum_j\int_0^{h_j} s^2 \gamma(s) \cK_j\big(\delta_q,(\tilde e_u)_t;s\big) ds 
 \leq  \ Ch^{k} \|(\tilde e_u)_t\|_{L^2(a,b)} \sum_j\int_0^{h_j} s^2\gamma(s) ds.
\end{aligned}
\end{equation}
As for the second term on the right hand side of \eqref{eq:asy_aux1}, invoking the definition of $\cK_j$ in \eqref{eq:hk}, we collect that
\begin{equation}\label{eq:asy_aux3}
\begin{aligned}
&  -2\sum_j\int_{h_j}^\delta s^2 \gamma(s) \int_{I_j} \frac{1}{s} \big(\delta_q(x+s,t) - \delta_q(x,t)\big) \big(\tilde e_u(x,t)\big)_t dx ds \\
\leq & \ \sum_j \int_{h_j}^\delta 4s^2\gamma(s) \|\delta_q\|_{L^\infty(a,b)} \|(\tilde e_u)_t\|_{L^\infty(a,b)} ds\\
\leq&\  Ch^{k} \|(\tilde e_u)_t\|_{L^2(a,b)} \sum_j \int_{h_j}^\delta s^2\gamma(s) ds.
\end{aligned}
\end{equation}
Plugging \eqref{eq:asy_aux2_1}--\eqref{eq:asy_aux3} into \eqref{eq:asy_aux1}, and using H\"older inequality yields
\begin{equation}\label{eq:asy_aux4}
-2\int_0^\delta s^2 \gamma(s) \cK\big(\delta_q, (\tilde e_u)_t;s\big) ds \leq Ch^{2k} +  \frac{1}{2}\|(\tilde e_u)_t\|_{L^2(a,b)}^2 .
\end{equation}
Now, combining \eqref{eq:asy_1aux2} for the case of $\delta < h_j$ and \eqref{eq:asy_aux4} for the case of $\delta \geq h_j$, we arrive at
\begin{equation}\label{eq:asy_aux4_1}
\int_0^\delta s^2 \gamma(s) \cK\big(\delta_q, (\tilde e_u)_t;s\big) ds \leq Ch^{2k} +  \frac{1}{2}\|(\tilde e_u)_t\|_{L^2(a,b)}^2 \quad \forall \delta > 0.
\end{equation}

Similarly, for the second term in \eqref{eq:asy_eq0}, one can derive that
\begin{equation}\label{eq:asy_aux5}
\int_0^\delta s^2 \gamma(s) \cH\big((\delta_u)_t, \tilde e_q;s\big) ds \leq Ch^{2k} +\int_0^\delta s^2\gamma(s) \|\tilde e_q(x,t;s)\|_{L^2(a,b)}^2ds \quad \forall \delta >0.
\end{equation}

Plugging \eqref{eq:asy_aux4_1}--\eqref{eq:asy_aux5} into the error equation \eqref{eq:error_aux3}, and using Gr\"onwall's inequality yields
\begin{equation}\label{eq:asy_aux8}
\frac{1}{2}\|(\tilde e_u)_t\|_{L^2(a,b)}^2 + \int_0^\delta s^2 \gamma(s) \|\tilde e_q(x,t;s)\|_{L^2(a,b)}^2  ds \leq   Ch^{2k}.
\end{equation}

\underline{\emph{Step two.}} We next proceed to estimate $\|\tilde e_u\|_{L^2(a,b)}$. To this end, we consider the error equation \eqref{eq:aux6}, and use the $L^2$-initial projection, the H\"older's inequality and the estimate \eqref{eq:asy_aux8} to find that
\begin{equation}\label{eq:asy_aux7}
\begin{aligned}
\frac{1}{2} \|\tilde e_u(x,\tau)\|_{L^2(a,b)}^2  &\leq \ Ch^{2k} + \frac{1}{8T^2}\|\cI_{\tilde e_u}(x,0)\|_{L^2(a,b)}^2 \\
&+ 2\!\int_0^\tau \!\!\!\int_0^\delta\!\! s^2 \gamma(s) \Big(\!-\mathcal{K}\big(\! {\delta_q}, \cI_{\tilde e_u}; s\big) + \mathcal{H}\big(\cI_{\delta_u}, {\tilde e_q}; s\big) \! \Big)ds dt.
\end{aligned}
\end{equation}
Moreover, following the similar analysis as the derivation of \eqref{eq:asy_aux4_1}-\eqref{eq:asy_aux5} in \emph{step one}, we obtain
\begin{equation}\label{eq:asy_aux10}
\begin{aligned}
& 2\!\int_0^\tau \!\!\!\int_0^\delta\!\! s^2 \gamma(s) \Big(\!-\mathcal{K}\big(\! {\delta_q}, \cI_{\tilde e_u}; s\big) + \mathcal{H}\big(\cI_{\delta_u}, {\tilde e_q}; s\big) \! \Big)ds dt \\
\leq & \ Ch^{2k} + \frac{1}{8T^3}\int_0^\tau \|\cI_{\tilde e_u}\|_{L^2(a,b)}^2 dt + \int_0^\tau\int_0^\delta s^2\gamma(s) \|\tilde e_q(x,t;s)\|_{L^2(a,b)}^2ds dt\\
\leq &\ Ch^{2k} + \frac{1}{8} \max_{t\in[0,T]} \|\tilde e_u(x,t)\|_{L^2(a,b)}^2,
\end{aligned}
\end{equation}
where the last inequality is obtained by using the relation \eqref{eq:aux7}, and the estimate \eqref{eq:asy_aux8}.

 Finally, plugging \eqref{eq:asy_aux10} into 
\eqref{eq:asy_aux7}, and invoking \eqref{eq:aux7}, we conclude that
\begin{equation}\label{eq:asy_aux14}
\max_{t\in[0,T]}\|\tilde e_u\|_{L^2(a,b)}^2 \leq Ch^{2k}, \quad k \geq 1.
\end{equation}

 \textbf{Case II:} $k=0$. 

In the following, we discuss the case where $k = 0$. Here, instead of $L^2$-projection defined in \eqref{eq:l2_projection}, we consider Gauss-Radau projection defined in \eqref{eq:gr_projection}. Denote the errors by
\[\begin{aligned}
\tilde e_u = e_u + \delta_u &:= \big(u_h - u\big) + \big(u - P_h^+ u\big),\\
\tilde e_q = e_q + \delta_q &:= \big(q_h - q\big) + \big(q - P_h^- q\big),
\end{aligned}\]
we then perform a two-step procedure to derive an estimate on $\|e_u\|_{L^2(a,b)}$.

\underline{\emph{Step one.}} Similar to Case I, to derive an estimate for $\|e_u\|_{L^2(a,b)}$, we first derive an estimate for $\|(\tilde e_u)_t\|_{L^2(a,b)}$. We note that the error equation, following the same derivation of \eqref{eq:error_aux3}, satisfies
\begin{equation}\label{eq:asy_aux18}
\begin{aligned}
& \frac{d}{dt}\bigg( \frac{1}{2}\|(\tilde e_u)_t\|_{L^2(a,b)}^2 + \int_0^\delta s^2 \gamma(s) \|\tilde e_q(x,t;s)\|_{L^2(a,b)}^2  ds\bigg) \\
= \ & -  2 \int_0^\delta s^2 \gamma(s) \Big(\mathcal{K}\big(\delta_q, (\tilde e_u)_t; s\big) + \mathcal{H}\big((\delta_u)_t, \tilde e_q; s\big)  \Big)ds\\
& - \int_a^b (\delta_u)_{tt} (\tilde e_u)_t dx - 2\int_0^\delta s^2\gamma(s) \int_a^b \delta_q(x,t;s)\tilde e_q(x,t;s)dxds, \\
\leq  \ & -  2 \int_0^\delta s^2 \gamma(s) \Big(\mathcal{K}\big(\delta_q, (\tilde e_u)_t; s\big) + \mathcal{H}\big((\delta_u)_t, \tilde e_q; s\big)  \Big)ds \\
& +  Ch^2 + \frac{1}{2}\|(\tilde e_u)_t\|_{L^2(a,b)}^2 + \int_0^\delta s^2\gamma(s) \|\tilde e_q(x,t;s)\|_{L^2(a,b)}^2ds,
\end{aligned}
\end{equation}
where the last inequality is derived by using H\"older's inequality and the property of the projection operator in \eqref{projre_gr}. To bound the term including $\cK$ and $\cH$ in \eqref{eq:asy_aux18}, we follow the same analysis conducted in \cite{du2019convergence} when $h_j = h$ (see their equations (4.29)-(4.33), we omit the derivation for simplicity) to obtain
\begin{equation}\label{eq:asy_aux20}
\begin{aligned}
& -2 \int_0^\delta s^2 \gamma(s) \Big(\mathcal{K}\big(\delta_q, (\tilde e_u)_t; s\big) + \mathcal{H}\big((\delta_u)_t, \tilde e_q; s\big)\Big) ds\\
\leq &\ Ch + \frac{1}{2}\|(\tilde e_u)_t\|_{L^2(a,b)}^2 + \int_0^\delta s^2 \gamma(s) \|\tilde e_q(x,t;s)\|_{L^2(a,b)}^2 ds.
\end{aligned}
\end{equation}
Then, plugging \eqref{eq:asy_aux20} into \eqref{eq:asy_aux18} and using Gr\"ownwall's inequality, we arrive at 
\begin{equation}\label{eq:asy_eut_bound}
\frac{1}{2}\|(\tilde e_u)_t\|_{L^2(a,b)}^2 + \int_0^\delta s^2 \gamma(s) \|\tilde e_q(x,t;s)\|_{L^2(a,b)}^2  ds \leq Ch.
\end{equation}

\underline{\emph{Step two.}} We next proceed to the estimation of $\|\tilde{e}_u\|_{L^2(a,b)}^2$. Similar to Case I, following the same derivation of \eqref{eq:aux6}, we find that the error equation  reads as
\begin{equation}\label{eq:asy_aux18_1}
\begin{aligned}
\frac{d}{dt} \bigg( \frac{1}{2}\|\tilde e_u\|_{L^2(a,b)}^2 & - \int_0^\delta s^2 \gamma(s) \|\cI_{\tilde e_q}(x,t;s)\|_{L^2(a,b)}^2 ds\bigg) 
=  -\int_a^b \frac{d}{dt} \big((\tilde e_u)_{t} \cI_{\tilde e_u}\big) dx  \\
&+ 2 \int_0^\delta s^2 \gamma(s) \Big(-\mathcal{K}\big({\delta_q}, \cI_{\tilde e_u}; s\big) + \mathcal{H}\big(\cI_{\delta_u}, \tilde e_q; s\big)  \Big)ds\\
&+\int_a^b (\delta_u)_{tt} \cI_{\tilde e_u} dx + 2\int_0^\delta s^2\gamma(s) \int_a^b \cI_{\delta_q}(x,t;s)\tilde e_q(x,t;s)dxds.
\end{aligned}
\end{equation}
Integrating \eqref{eq:asy_aux18_1} over $(0,\tau)$ with respect to time variable, we then collect
\begin{equation}\label{eq:asy_aux19}
\begin{aligned}
&\frac{1}{2} \|\tilde e_u(x,\tau)\|_{L^2(a,b)}^2 - \frac{1}{2} \|\tilde e_u(x,0)\|_{L^2(a,b)}^2 + \int_0^\delta s^2 \gamma(s) \|\cI_{\tilde e_q}(x,0;s)\|_{L^2(a,b)}^2ds\\
=\  &\!\! \int_a^b\!\!\! \big(\tilde e_u(x,0)\big)_t \cI_{\tilde e_u}(x,0)dx + 2\!\int_0^\tau \!\!\!\int_0^\delta\!\! s^2 \gamma(s) \Big(\!-\mathcal{K}\big({\delta_q}, \cI_{\tilde e_u}; s\big) + \mathcal{H}\big(\cI_{\delta_u}, {\tilde e_q}; s\big) \! \Big)ds dt \\
+ & \int_0^\tau \int_a^b (\delta_u)_{tt} \cI_{\tilde e_u} dxdt + 2\int_0^\tau\int_0^\delta s^2\gamma(s) \int_a^b \cI_{\delta_q}(x,t;s)\tilde e_q(x,t;s)dxds dt \\
\leq \ & 2\!\int_0^\tau \!\!\!\int_0^\delta\!\! s^2 \gamma(s) \Big(\!-\mathcal{K}\big({\delta_q}, \cI_{\tilde e_u}; s\big) + \mathcal{H}\big(\cI_{\delta_u}, {\tilde e_q}; s\big) \! \Big)ds dt + Ch + \frac{1}{8}\max_{t\in[0,T]} \|\tilde e_u\|_{L^2(a,b)}^2,
\end{aligned}
\end{equation}
where the last inequality is derived by using Gauss-Radau initial projection with $u_h(x,0) = P_h^+u(x,0)$ and $q_h(x,0; s) = P_h^-q(x,0; s)$, the estimate \eqref{eq:asy_eut_bound}, the property of the Gauss-Radau projection in \eqref{projre_gr}, the H\"older's inequality, and the relation between error and time-integrated error shown in \eqref{eq:aux7}. To bound the term involving $\cK$ and $\cH$ in \eqref{eq:asy_aux19}, we follow the same process as the derivation of \eqref{eq:asy_aux20} to find that
\begin{equation}\label{eq:asy_aux_19_1}
\begin{aligned}
&2 \int_0^\tau\int_0^\delta s^2 \gamma(s) \Big(-\mathcal{K}\big({\delta_q}, \cI_{\tilde e_u}; s\big) + \mathcal{H}\big(\cI_{\delta_u}, \tilde e_q; s\big)  \Big)dsdt \\
\leq \ &Ch + \frac{1}{8T^3}\int_0^\tau\|\cI_{\tilde e_u}\|_{L^2(a,b)}^2 dt + \int_0^\tau \int_0^\delta s^2 \gamma(s)\|\tilde e_q(x,t;s)\|_{L^2(a,b)}^2 dsdt \\
\leq \ &Ch + \frac{1}{8} \max_{t\in[0,T]}\|\tilde e_u\|_{L^2(a,b)}^2,
\end{aligned}
\end{equation}
where the relations \eqref{eq:aux7} and \eqref{eq:asy_eut_bound} have been used to derive the last inequality. Plugging \eqref{eq:asy_aux_19_1} into \eqref{eq:asy_aux19} then yields
\begin{equation}\label{eq:asy_bound}
\max_{t\in[0,T]} \|\tilde e_u\|_{L^2(a,b)}^2 \leq Ch.
\end{equation}

 Finally, combine \eqref{eq:asy_aux8}, \eqref{eq:asy_aux14}, \eqref{eq:asy_eut_bound}, \eqref{eq:asy_bound} and invoke triangle inequality and the property of the Gauss-Radau projection, we obtain the desired estimate \eqref{eq:asy_theorem}.
\end{proof}

\begin{remark}
For the asymptotic compatibility of the DG scheme proposed in Theorem \ref{theorem_asy}, we can only show suboptimal convergence for a general kernel satisfying \eqref{eq:kernel}. However, numerical results consistently show optimal convergence, as shown in Section \ref{sec:convergence}. A thorough theoretical study of this phenomenon is left to future work.
\end{remark}

\section{Time discretization}
\label{sec:time}

In this section, we extend the semi-discrete DG scheme to the fully discrete method, which preserves the fully discrete energy.
	
\subsection{Crank--Nicolson time discretization}\label{sec_cn}

We study the Crank--Nicolson scheme for time discretization and show the energy conservation properties of the corresponding fully discrete scheme.  Let $0 = t_0 < t_1 < \cdots < t_n < \cdots < \cdots < t_{N_t} = T$ and denote $h_t:= t_{n+1} - t_n$.  Here we use the uniform time step $h_t$ and denote the DG solution at $t = t_n$ by $u_h^n$ and $q_h^n$. We also introduce the following three operators, which will be used in the rest of the content
\[\texttt{mean value operator}: {\mathcal{M}} u^n := \frac{u^{n+1} + u^{n-1}}{2},\]
and
\[\begin{aligned}
&\texttt{first-order central difference operator}: \mathcal{D} u^n := \frac{u^{n+1} - u^{n-1}}{2h_t}, \\
&\texttt{second-order central difference operator}: \mathcal{D}^2 u^n := \frac{u^{n+1} - 2u^n + u^{n-1}}{h_t^2}. 
\end{aligned}\]

With these notations, the fully discrete approximation $(u^n_h, q^n_h) = \big(u_h(\cdot, t_n), q_h(\cdot, t_n)\big)$ of problem \eqref{eq:NW3} then satisfies
\begin{equation}\label{eq:dis_DG}
\begin{aligned}
\int_{I_j} \mathcal{D}^2 u_h  v_h d{x} - 2 \int^\delta_0 s^2 \gamma(s) \cK(\mathcal{M} q_h, v_h; s) ds &= 0, \\
        \int_{I_j} q^{n+1}_h w_h d{x} - \cH (u^{n+1}_h, w_h; s) &= 0, \\
        \int_{I_j} q^{n-1}_h w_h d{x} - \cH (u^{n-1}_h, w_h; s) &= 0,
\end{aligned}
\end{equation}
for all test functions $v_h, w_h\in V_h^k$. Moreover, we have the following conservation property.
\begin{theorem}\label{dis_thm}
For all $n$, the solution to the fully discrete DG scheme \eqref{eq:dis_DG} conserves the discrete energy
\begin{equation}\label{discrete_energy}
E_h^{n+1} := \Big\vert\!\Big\vert \frac{u_h^{n+1} - u_h^n}{h_t} \Big\vert\!\Big\vert_{L^2(a,b)}^2 + \int_0^\delta \Vert q_h^{n+1} \Vert_{L^2(a,b)}^2 + \Vert q_h^{n} \Vert_{L^2(a,b)}^2 ds.
\end{equation}
\end{theorem}
\begin{proof}
To prove the fully discrete energy conservation \eqref{discrete_energy}, we choose the test function $v_h = \mathcal{D} u_h^n$ in the first equation of \eqref{eq:dis_DG} and sum the resulting equation over all subcells $I_j$ to obtain
\begin{equation}\label{fully1}
		\int_a^b (\mathcal{D}^2 u_h^n) (\mathcal{D} u_h^n) \ dx - 2\int_0^\delta s^2\gamma(s) \cK(\mathcal{M} q_h^n, \mathcal{D} u_h^n; s) ds = 0. 
\end{equation}

Choose the test function $w_h = \mathcal{M} q_h^{n}/h_t$ in the second and the third equation of \eqref{eq:dis_DG}, and subtract the resulting two equations, then sum them over all subcells to collect
\begin{equation}\label{fully2}
	\int_a^b 2 (\mathcal{D} q^{n}_h) (\mathcal{M}q_h^{n})\ d{x} - 2\cH (\mathcal{D} u^{n}_h, \mathcal{M} q_h^{n}; s) = 0.
\end{equation}
Next, multiply \eqref{fully2} by $s^2\gamma(s)$ and take the integral over $(0,\delta)$ in $s$, and sum the resulting equation with \eqref{fully1}, and invoke Lemma \ref{lemma1}, we arrive at
\begin{equation*}
\int_a^b (\mathcal{D}^2 u_h^n) (\mathcal{D} u_h^n) \ dx + 2\int_0^\delta s^2\gamma(s) \int_a^b (\mathcal{D} q^{n}_h) (\mathcal{M} q_h^{n})\  d{x} ds = 0,
\end{equation*}
which is equivalent to
\begin{equation*}
\begin{aligned}
& \Big\Vert \frac{u^{n+1}_h - u^n_h}{h_t} \Big\Vert_{L^2(a,b)}^2  + \int_0^\delta s^2 \gamma(s) \big( \| q^{n+1}_h \|_{L^2(a,b)}^2 + \| q^n_h \|_{L^2(a,b)}^2 \big) ds  \\
= \ & \Big\Vert  \frac{u^n_h - u^{n-1}_h}{h_t} \Big\Vert_{L^2(a,b)}^2 + \int_0^\delta s^2 \gamma(s) \big( \| q^n_h \|_{L^2(a,b)}^2 + \| q^{n-1}_h \|_{L^2(a,b)}^2 \big) ds.
 \end{aligned}
    \end{equation*}
    That is, we have $E_h^{n+1} = E_h^n$ for all $n$.
\end{proof}

\section{Numerical simulations}
\label{sec:numerical}

In this section, we present several numerical experiments to illustrate and support the theoretical studies of the proposed DG scheme. Throughout these studies, we use a standard modal basis formulation, and consider the kernel function $\gamma(s) = \frac{3 - \alpha}{2 \delta^{3 - \alpha}} |s|^{-\alpha}, 0<\alpha<3$. It is straightforward to verify that $\int^\delta_{-\delta} s^2 \gamma(s) ds = 1$ and $\gamma(s)$ is integrable when $0 < \alpha < 1$. We report numerical errors in the $L^2$-norm, which are computed as follows:
{\begin{equation*}\label{eq:errors}
e_u := \left\| u({x}, T) - u_h({x}, T)\right\|_{L^2(a,b)} = \Big(\sum_j\sum_{i=0}^{k+2} \omega_i^j\big(u({x}_i^j, T) - u_h({x}_i^j, T)\big)^2 \Big)^{1/2},
\end{equation*}
where ${x}_i^j$ are the Gauss-Labatto quadrature nodes in $I_j$ and $\omega_i^j$ are the corresponding weights.



\subsection{Convergence study}\label{sec:convergence}

We consider the following nonlocal wave equation
\begin{equation}\label{eq:prob1}
 u_{tt}(x,t) + \cL_\delta u(x,t) = f_\delta(x,t),\quad (x,t) \in (0,1)\times(0,1],
\end{equation}
with $f_\delta (x,t) = - \cos(2\pi t) \left( 4\pi^2\sin(2\pi x) +2\int_{-\delta}^\delta \gamma(s) \big( \sin(2\pi (x+s) ) - \sin(2\pi x) \big) ds \right)$. Given the initial conditions $u(x,0) = \sin(2\pi x)$ and $u_t(x,0) = 0$ together with periodic boundary conditions, the exact solution of \eqref{eq:prob1} is $u(x,t)=\cos(2\pi t)\sin(2\pi x)$.

In the numerical experiments, we uniformly discretize the spatial interval with vertices $x_j = jh$, $j = 0, \cdots,N$, where $h = 1/N$, and choose the time step size $h_t = 2\times 10^{-5}$ to guarantee that the temporal discretization error is dominated by the spatial discretization error. Throughout the study we consider the degree of the approximation space for $u_h$ to be $k \in \{0, 1, 2\}$, and consider cases where the horizon $\delta$ is fixed or $\delta/h$ is fixed.

\begin{table}[!htb]
 	\footnotesize
 \begin{center}
 		\scalebox{0.85}{
 		\begin{tabular}{c| c c c c c c c c c c c c c c}
 				\hline
 				 ~ & ~ & ~ & $\vert$ & $\alpha = 1/4$ & ~ & ~ & $\alpha = 1/2$ & ~ & ~ & $\alpha = 3/2$ & ~ &  ~& $\alpha = 5/2$ & ~ \\
 				  \cline{5-6} \cline{8-9} \cline{11-12} \cline{14-15}
     ~ & $k$ & $N$ & $\vert$ & $e_u$ & $\text{order}$ & ~ & $e_u$ & $\text{order}$ & ~ & $e_u$ & $\text{order}$ & ~ & $e_u$ & $\text{order}$  \\
 				\cline{1-15}
~ & ~ & 10 & $\vert$ &1.2721e-01 &-- & & 1.2721e-01 &   --  & ~ & 1.2721e-01 &   --   & ~ &1.2721e-01 &   --  \\
~ & ~ & 20 & $\vert$ &6.3996e-02 &0.9911 & & 6.3996e-02 & 0.9911& ~ & 6.3996e-02 & 0.9911 & ~ &6.3996e-02 & 0.9911 \\
~ & 0  & 40 & $\vert$ &3.2047e-02 &0.9978& & 3.2047e-02 & 0.9978& ~ & 3.2047e-02 & 0.9978 & ~ &3.2047e-02 & 0.9978\\
~ & ~ & 80 & $\vert$ &1.6030e-02& 0.9994 & & 1.6030e-02 & 0.9994& ~ & 1.6030e-02 & 0.9994 & ~ &1.6030e-02 & 0.9994 \\
 				    ~ & ~ & ~ & ~ & ~  & ~ & ~ & ~ & ~ & ~ & ~ & ~ & ~ & ~ \\
~ &~  & 10 & $\vert$ &1.0335e-02 &-- & & 1.0335e-02 &    --  & ~ & 1.0335e-02 &    --   & ~ &1.0335e-02 &    --  \\
~ &  ~  & 20 & $\vert$ &2.5966e-03 &1.9929 & & 2.5966e-03 & 1.9929& ~ & 2.5966e-03 & 1.9929 & ~ &2.5966e-03 & 1.9929  \\
$\delta = 10^{-5}$&1 & 40 & $\vert$ &6.4995e-04 &1.9982& & 6.4995e-04 & 1.9982& ~ & 6.4995e-04 & 1.9982& ~ &6.4995e-04 & 1.9982\\
~ & ~  & 80 & $\vert$ &1.6254e-04 &1.9996 & & 1.6254e-04 & 1.9996& ~ & 1.6254e-04 & 1.9996 & ~ &1.6254e-04 & 1.9996 \\
 		~ & ~ & ~ & ~ & ~ & ~ & ~ & ~ & ~ & ~ & ~ & ~ & ~ & ~ & ~\\
~ & ~  & 10 & $\vert$ &5.4954e-04 &--& & 5.4954e-04 &      --  & ~ & 5.4954e-04 &     --  & ~  & 5.4954e-04 &     --  \\
~ & ~  & 20 & $\vert$ &6.8965e-05 &2.9943 & & 6.8965e-05 & 2.9943& ~ & 6.8965e-05 & 2.9943& ~ & 6.8965e-05 & 2.9943  \\
~ & 2   & 40 & $\vert$ &8.6292e-06 &2.9986& & 8.6292e-06 & 2.9986& ~ & 8.6292e-06 & 2.9986 & ~ & 8.6292e-06 & 2.9986\\
~ & ~  & 80 & $\vert$ &1.0789e-06 & 2.9996 & & 1.0790e-06 & 2.9996 & ~ & 1.0789e-06 & 2.9996 & ~ &1.0789e-06 & 2.9996 \\
 				\hline
     \hline
     ~ & ~ & 10 & $\vert$ &1.2721e-01 &-- & & 1.2721e-01&  --  & ~ & 1.2721e-01&  --   & ~ &1.2721e-01 &  --  \\
~ & ~ & 20 & $\vert$ &6.3996e-02 &0.9911 & & 6.3996e-02&  0.9911& ~ & 6.3996e-02&  0.9911 & ~ &6.3996e-02 &  0.9911 \\
~ & 0  & 40 & $\vert$ &3.2047e-02 &0.9978& & 3.2047e-02&  0.9978& ~ & 3.2047e-02&  0.9978 & ~ &3.2047e-02  &0.9978\\
~ & ~ & 80 & $\vert$ &1.6030e-02 & 0.9994 & & 1.6030e-02&  0.9994& ~ & 1.6030e-02&  0.9994 & ~ &1.6030e-02  &  0.9994 \\
 				    ~ & ~ & ~ & ~ & ~  & ~ & ~ & ~ & ~ & ~ & ~ & ~ & ~ & ~ \\
~ &~  & 10 & $\vert$ &1.0335e-02 &-- & & 1.0335e-02&   --  & ~ & 1.0335e-02&   --   & ~ &1.0335e-02 &   --  \\
~ &  ~  & 20 & $\vert$ &2.5966e-03 &1.9929 & & 2.5966e-03&   1.9929& ~ &2.5966e-03&  1.9929 & ~ &2.5966e-03 &   1.9929 \\
$\delta = 0.2$&1 & 40 & $\vert$ &6.4995e-04 &1.9982 & & 6.4995e-04&   1.9982& ~ & 6.4995e-04&   1.9982& ~ &6.4995e-04  &1.9982\\
~ & ~  & 80 & $\vert$ &1.6254e-04 &1.9996 & & 1.6254e-04& 1.9996& ~ & 1.6254e-04&   1.9996 & ~ &1.6254e-04  &  1.9996 \\
 		~ & ~ & ~ & ~ & ~ & ~ & ~ & ~ & ~ & ~ & ~ & ~ & ~ & ~ & ~\\
~ & ~  & 10 & $\vert$ &5.4954e-04 &--& & 5.4954e-04&     --  & ~ & 5.4954e-04&    --  & ~  & 5.4957e-04 &    --  \\
~ & ~  & 20 & $\vert$ &6.8965e-05 &2.9943 & & 6.8965e-05&     2.9943& ~ & 6.8965e-05&    2.9943& ~ &6.8979e-05&    2.9941  \\
~ & 2   & 40 & $\vert$ &8.6292e-06 &2.9986& & 8.6292e-06&     2.9986& ~ & 8.6292e-06&    2.9986 & ~ & 8.6308e-06 &   2.9986\\
~ & ~  & 80 & $\vert$ &1.0789e-06 &2.9996 & &1.0789e-06 &2.9996      & ~ &1.0789e-06 &2.9996     & ~ & 1.0791e-06 & 2.9997  \\
 				\hline
 \end{tabular}
 		}
 \end{center}
\caption{\scriptsize{$L^2$-errors and corresponding convergence rates for $u$ in the problem \eqref{eq:prob1} when the nonlocal horizon $\delta = 10^{-5}$ or $\delta = 0.2$, using $\mathcal{P}^k$ polynomials with $k=0,1,2$. The time step size is $h_t=2\times10^{-5}$ with the terminal time $T=1$, and the spatial interval is divided into $N$ uniform cells. These results show that optimal convergence for both integrable kernel with $0<\alpha<1$, and non-integrable kernel with $\alpha\geq 1$.}}\label{tab:convergence_fixed_delta}
 \end{table}

 \begin{table}[!htb]
 	\footnotesize
 \begin{center}
 		\scalebox{0.85}{
 		\begin{tabular}{c| c c c c c c c c c c c c c c}
 				\hline
 				 ~ & ~ & ~ & $\vert$ & $\alpha = 1/4$ & ~ & ~ & $\alpha = 1/2$ & ~ & ~ & $\alpha = 3/2$ & ~ &  ~& $\alpha = 5/2$ & ~ \\
 				  \cline{5-6} \cline{8-9} \cline{11-12} \cline{14-15}
     ~ & $k$ & $N$ & $\vert$ & $e_u$ & $\text{order}$ & ~ & $e_u$ & $\text{order}$ & ~ & $e_u$ & $\text{order}$ & ~ & $e_u$ & $\text{order}$  \\
 				\cline{1-15}
~ & ~ & 10 & $\vert$ &1.2721e-01 &-- & & 1.2721e-01 &   --  & ~ & 1.2721e-01 &   --   & ~ &1.2721e-01 &   --  \\
~ & ~ & 20 & $\vert$ &6.3996e-02 &0.9911 & & 6.3996e-02 & 0.9911& ~ & 6.3996e-02 & 0.9911 & ~ &6.3996e-02 & 0.9911 \\
~ & 0  & 40 & $\vert$ &3.2047e-02 &0.9978& & 3.2047e-02 & 0.9978& ~ & 3.2047e-02 & 0.9978 & ~ &3.2047e-02 & 0.9978\\
~ & ~ & 80 & $\vert$ &1.6030e-02& 0.9994 & & 1.6030e-02 & 0.9994& ~ & 1.6030e-02 & 0.9994 & ~ &1.6030e-02 & 0.9994 \\
 				    ~ & ~ & ~ & ~ & ~  & ~ & ~ & ~ & ~ & ~ & ~ & ~ & ~ & ~ \\
~ &~  & 10 & $\vert$ &1.0335e-02 &-- & & 1.0335e-02 &    --  & ~ & 1.0335e-02 &    --   & ~ &1.0335e-02 &    --  \\
~ &  ~  & 20 & $\vert$ &2.5966e-03 &1.9929 & & 2.5966e-03 & 1.9929& ~ & 2.5966e-03 & 1.9929 & ~ &2.5966e-03 & 1.9929  \\
$\delta = h$&1 & 40 & $\vert$ &6.4995e-04 &1.9982& & 6.4995e-04 & 1.9982& ~ & 6.4995e-04 & 1.9982& ~ &6.4995e-04 & 1.9982\\
~ & ~  & 80 & $\vert$ &1.6254e-04 &1.9996 & & 1.6254e-04 & 1.9996& ~ & 1.6254e-04 & 1.9996 & ~ &1.6254e-04 & 1.9996 \\
 		~ & ~ & ~ & ~ & ~ & ~ & ~ & ~ & ~ & ~ & ~ & ~ & ~ & ~ & ~\\
~ & ~  & 10 & $\vert$ &5.4954e-04 &--& & 5.4954e-04 &      --  & ~ & 5.4954e-04 &     --  & ~  & 5.4954e-04 &     --  \\
~ & ~  & 20 & $\vert$ &6.8965e-05 &2.9943 & & 6.8965e-05 & 2.9943& ~ & 6.8965e-05 & 2.9943& ~ & 6.8965e-05 & 2.9943  \\
~ & 2   & 40 & $\vert$ &8.6292e-06 &2.9986& & 8.6292e-06 & 2.9986& ~ & 8.6292e-06 & 2.9986 & ~ & 8.6292e-06 & 2.9986\\
~ & ~  & 80 & $\vert$ &1.0789e-06 & 2.9996 & & 1.0789e-06 & 2.9996 & ~ & 1.0789e-06 & 2.9996 & ~ &1.0789e-06 & 2.9996 \\
 				\hline
     \hline
    ~ & ~ & 10 & $\vert$ &1.2721e-01 &-- & & 1.2721e-01 &   --  & ~ & 1.2721e-01 &   --   & ~ &1.2721e-01 &   --  \\
~ & ~ & 20 & $\vert$ &6.3996e-02 &0.9911 & & 6.3996e-02 & 0.9911& ~ & 6.3996e-02 & 0.9911 & ~ &6.3996e-02 & 0.9911 \\
~ & 0  & 40 & $\vert$ &3.2047e-02 &0.9978& & 3.2047e-02 & 0.9978& ~ & 3.2047e-02 & 0.9978 & ~ &3.2047e-02 & 0.9978\\
~ & ~ & 80 & $\vert$ &1.6030e-02& 0.9994 & & 1.6030e-02 & 0.9994& ~ & 1.6030e-02 & 0.9994 & ~ &1.6030e-02 & 0.9994 \\
 				    ~ & ~ & ~ & ~ & ~  & ~ & ~ & ~ & ~ & ~ & ~ & ~ & ~ & ~ \\
~ &~  & 10 & $\vert$ &1.0335e-02 &-- & & 1.0335e-02 &    --  & ~ & 1.0335e-02 &    --   & ~ &1.0335e-02 &    --  \\
~ &  ~  & 20 & $\vert$ &2.5966e-03 &1.9929 & & 2.5966e-03 & 1.9929& ~ & 2.5966e-03 & 1.9929 & ~ &2.5966e-03 & 1.9929  \\
$\delta = 3h$&1 & 40 & $\vert$ &6.4995e-04 &1.9982& & 6.4995e-04 & 1.9982& ~ & 6.4995e-04 & 1.9982& ~ &6.4995e-04 & 1.9982\\
~ & ~  & 80 & $\vert$ &1.6254e-04 &1.9996 & & 1.6254e-04 & 1.9996& ~ & 1.6254e-04 & 1.9996 & ~ &1.6254e-04 & 1.9996 \\
 		~ & ~ & ~ & ~ & ~ & ~ & ~ & ~ & ~ & ~ & ~ & ~ & ~ & ~ & ~\\
~ & ~  & 10 & $\vert$ &5.4954e-04 &--& & 5.4954e-04 &      --  & ~ & 5.4954e-04 &     --  & ~  & 5.4954e-04 &     --  \\
~ & ~  & 20 & $\vert$ &6.8965e-05 &2.9943 & & 6.8965e-05 & 2.9943& ~ & 6.8965e-05 & 2.9943& ~ & 6.8965e-05 & 2.9943  \\
~ & 2   & 40 & $\vert$ &8.6292e-06 &2.9986& & 8.6292e-06 & 2.9986& ~ & 8.6292e-06 & 2.9986 & ~ & 8.6292e-06 & 2.9986\\
~ & ~  & 80 & $\vert$ &1.0789e-06 & 2.9996 & & 1.0789e-06 & 2.9996 & ~ & 1.0789e-06 & 2.9996 & ~ &1.0789e-06 & 2.9996 \\
 				\hline
 \end{tabular}
 		}
 \end{center}
\caption{\scriptsize{$L^2$-errors and corresponding convergence rates for $u$ in the problem \eqref{eq:prob1} when $\delta = h$ or $\delta = 3h$, using $\mathcal{P}^k$ polynomials with $k=0,1,2$. The time step size is $h_t=2\times10^{-5}$ with the terminal time $T=1$, and the spatial interval is divided into $N$ uniform cells. These results show that optimal convergence for both integrable kernel with $0<\alpha<1$, and non-integrable kernel with $\alpha\geq 1$.}}\label{tab:convergence_fixed_m}
 \end{table}

 Table \ref{tab:convergence_fixed_delta} and Table \ref{tab:convergence_fixed_m} show the $L^2$-errors for $u$ with different number of elements $N=10, 20, 40, 80$. In Table \ref{tab:convergence_fixed_delta} the nonlocal horizon is fixed at $\delta = 10^{-6}$ and $\delta = 0.2$, while in Table \ref{tab:convergence_fixed_m} we set $\delta = h$ or $\delta = 3h$. The results show that the proposed scheme consistently achieves robust and optimal $(k+1)$-th order accuracy over different values of $\alpha$, which measures the singularity of the kernel function $\gamma(s)$. However, we note that when $\gamma(s)$ is not integrable, only suboptimal convergence can be proved, as discussed in Theorem \ref{theorem_asy}. A more detailed theoretical analysis of optimal convergence in this scenario is left to future work.

\subsection{Energy conservation}\label{sec:energy}

In this section, we test the energy conservative property of the proposed fully discrete DG scheme \eqref{eq:dis_DG}. To this end, we consider the following problem
\begin{equation}\label{eg:example2}
u_{tt}(x,t) + \cL_\delta(x,t) = 0, \quad (x,t) \in (0,1) \times (0, 1000],
\end{equation}
with periodic boundary conditions and initial data $u(x,0) = \sin(2\pi x), u_t(x,0) = 0$. We uniformly discretize the spatial interval with vertices $x_j = jh$, where $j = 0, \cdots,80$ and $h = 1/80$, and choose the time step size $h_t = 0.1$.

\begin{figure}
\centering
\includegraphics[width = 0.45\textwidth]{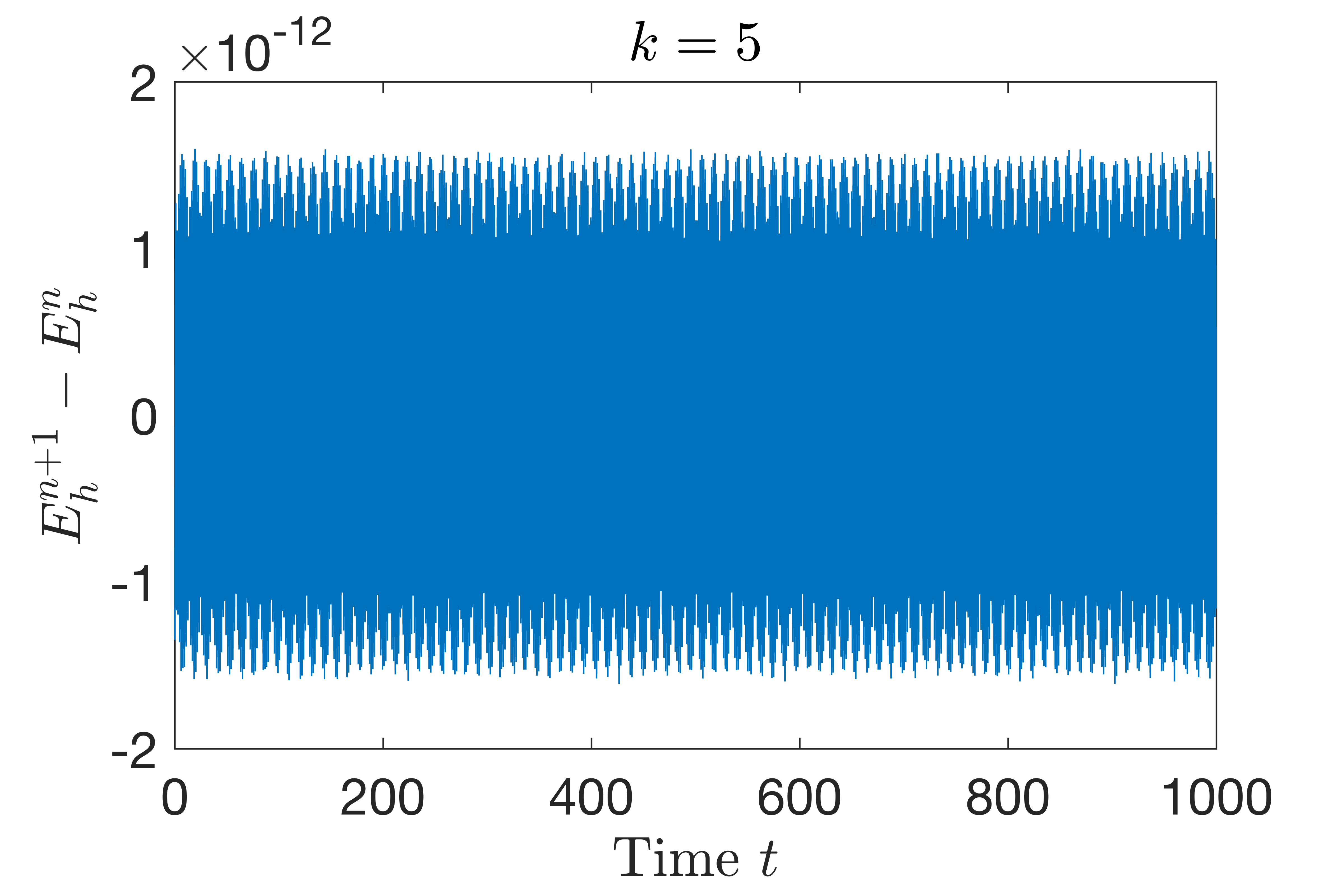}
\includegraphics[width = 0.45\textwidth]{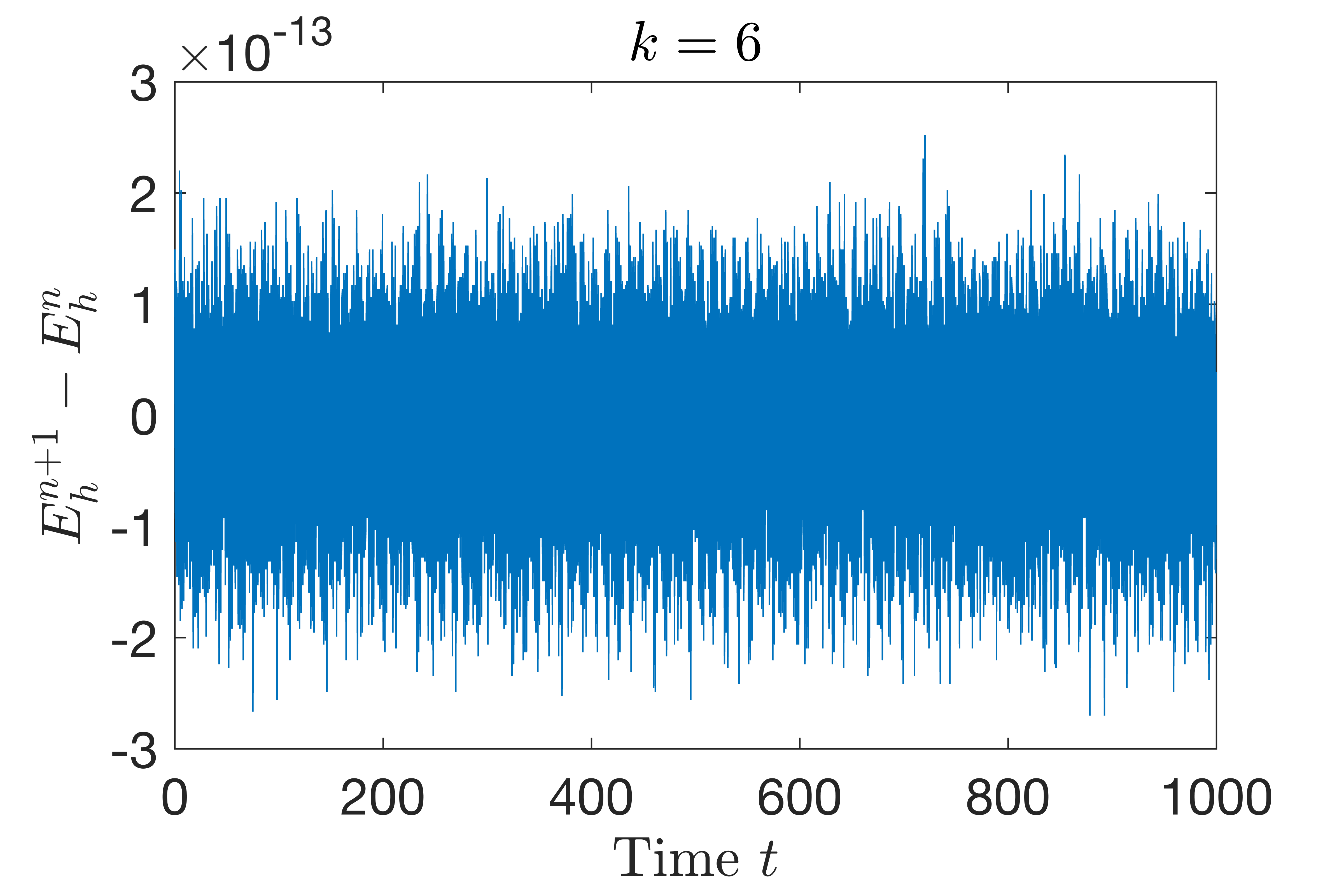}\\
\includegraphics[width = 0.45\textwidth]{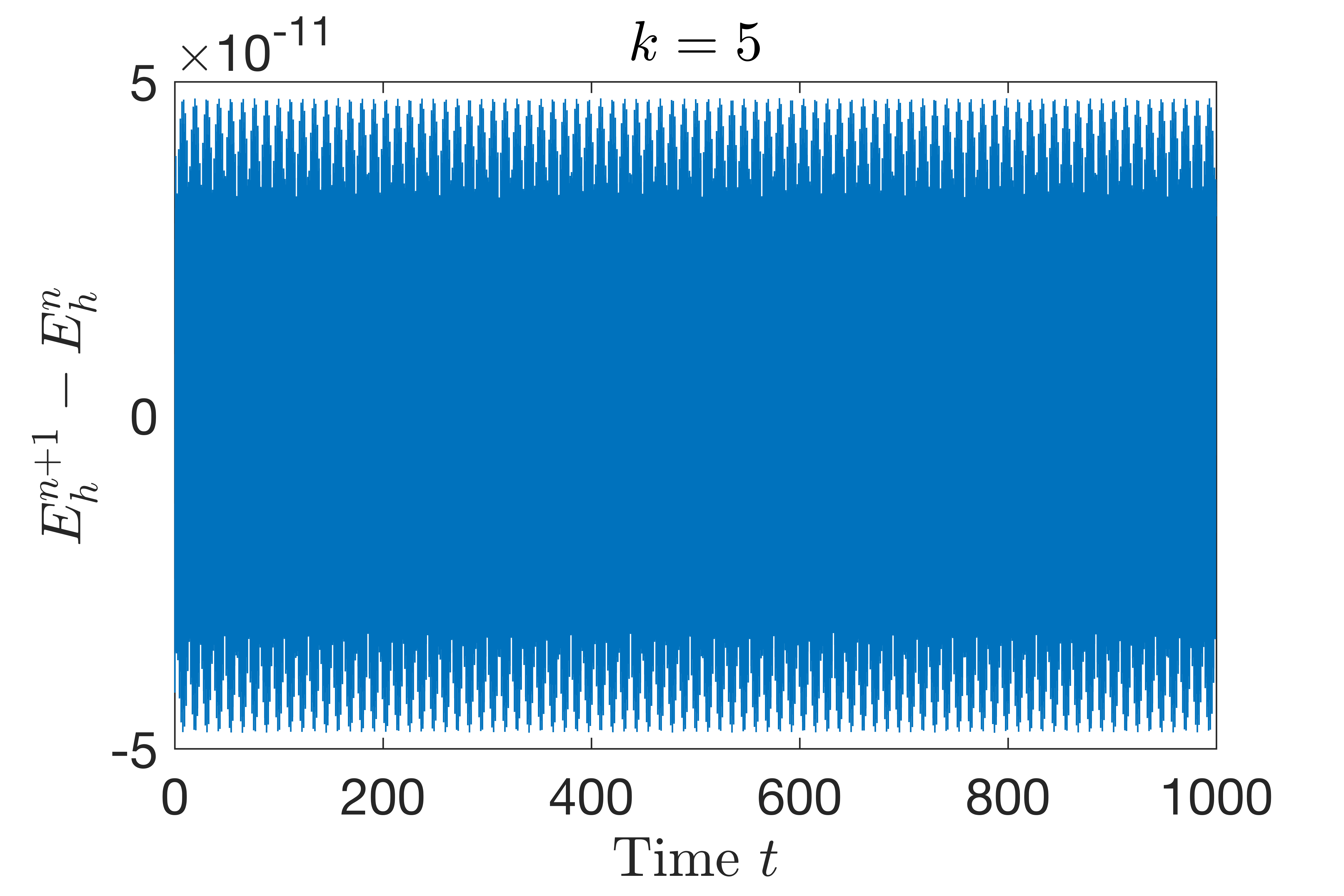}
\includegraphics[width = 0.45\textwidth]{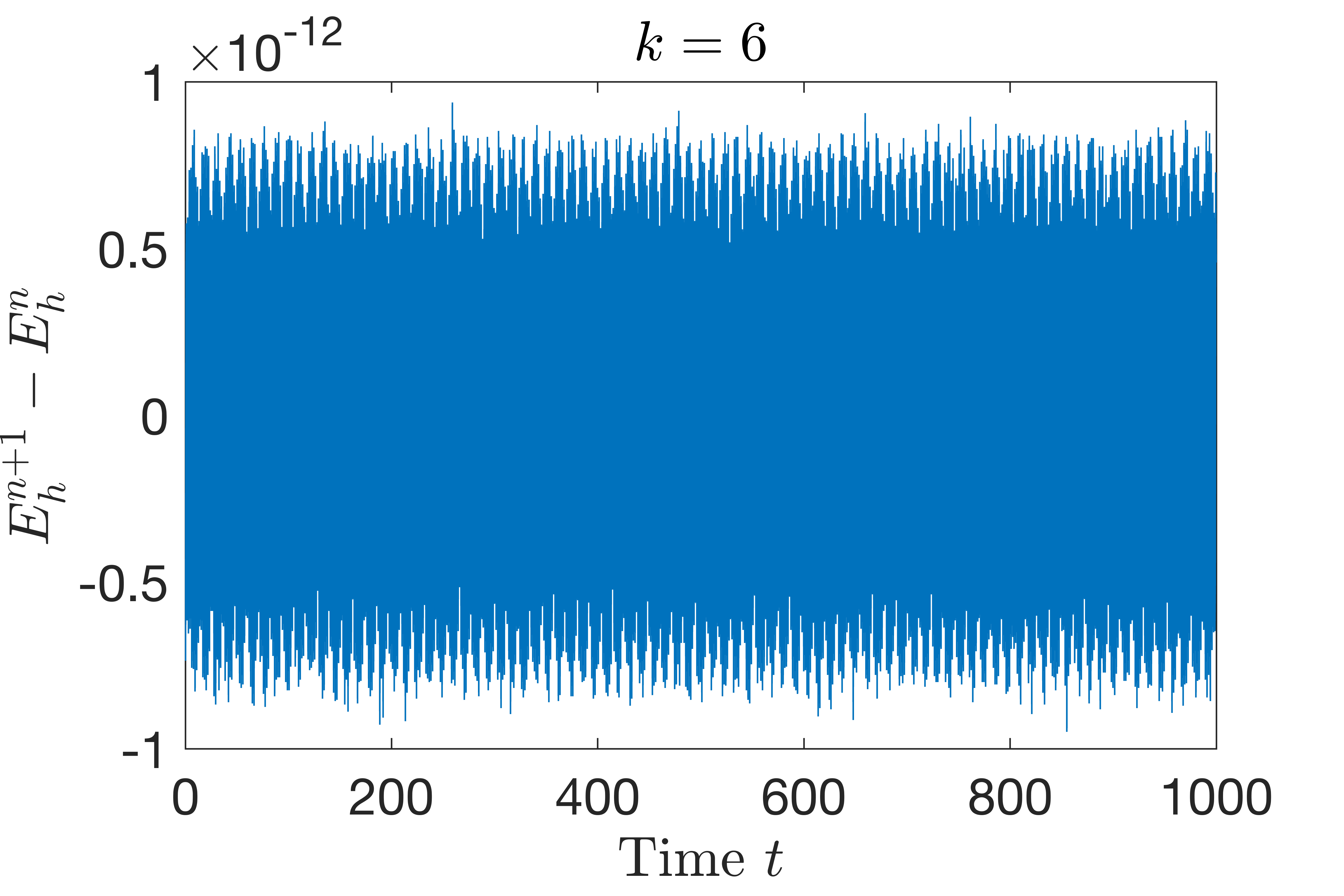}
\caption{\scriptsize{The trajectory of the numerical energy, as defined by \eqref{discrete_energy}, for the problem \eqref{eg:example2}. $\textbf{Top:}$ Result with integrable kernel $(\alpha = 2/3)$. $\textbf{Bottom:}$ Result with non-integrable kernel $(\alpha = 3/2)$ \textbf{Left:} Result with the degree of approximation being $k=5$. \textbf{Right:} Result with the degree of approximation being $k=6$.}} \label{fig:energy_history}
\end{figure}

 Figure \ref{fig:energy_history} shows the numerical energy trajectories for integrable kernel $\gamma(s)$ with $\alpha = 2/3$ (top) and non-integrable kernel $\gamma(s)$ with $\alpha = 3/2$ (bottom). The left plot shows the results for a degree of approximation of $k=5$, while the right plot corresponds to $k=6$. In all tests the nonlocal horizon is set to be $\delta = 0.025 = 2h$.

 We observe that the numerical energy is well preserved in all cases: about 12 digits of accuracy for $k=5$ and 13 digits for $k=6$ with an integrable kernel. For the non-integrable kernel, the energy conservation holds up to about 11 digits for $k=5$ and 12 digits for $k=6$. The slight deviation from machine precision (double precision is used) is due to numerical integration errors in evaluating the integrals in $E_h$. Overall, we find that numerical energy conservation improves as $k$ increases, consistent with the fact that higher-order methods have lower dispersion and dissipation errors compared to their lower-order counterparts.

\subsection{Limitting behaviour}

In this section we numerically demonstrate the behaviour of the proposed DG scheme for the nonlocal wave equations as $\delta \rightarrow 0$. We consider the problem \eqref{eg:example2} with periodic boundary conditions and initial data 
\begin{equation}\label{eq:initial_data_experi}
u(x,0) = \sin(2\pi x) \quad \text{and} \quad u_t(x,0) = 0.
\end{equation}
Let further the solution of the local problem \eqref{EQ:wave} with the initial data \eqref{eq:initial_data_experi} be denoted by $u_{\text{loc}}$, which we will use as the target solution. We then solve the problem \eqref{EQ:NW} numerically using the proposed DG scheme \eqref{EQ:DG_scheme}, with the degree of the approximation space set to be $k = 2$, a spatial discretisation of $h = 0.025$ and a time step of $h_t = 0.1$. The simulation is run to a final time of $T = 100$, with the horizon $\delta$ varying between ${3h, 2h, h, 10^{-3}, 10^{-4}, 10^{-5}}$ to mimic the limit behaviour as $\delta \rightarrow 0$. 

 Figure \ref{fig:asy_loc} presents the difference between the numerical solution of \eqref{EQ:NW}, denoted by $u_h$, and the local solution $u_{\text{loc}}$. The left panel shows the results for an integrable kernel with $\alpha = \frac{1}{2}$, while the right panel shows the results for a non-integrable kernel with $\alpha = \frac{3}{2}$. In both cases we observe that as $\delta$ decreases the numerical solution $u_h$ of the nonlocal problem converges to the local solution $u_{\text{loc}}$. This demonstrates the asymptotic compatibility of the proposed DG scheme for nonlocal wave equations and confirms its accuracy for both integrable and nonintegrable kernels.

\begin{figure}
\centering
\includegraphics[width = 0.45\textwidth]{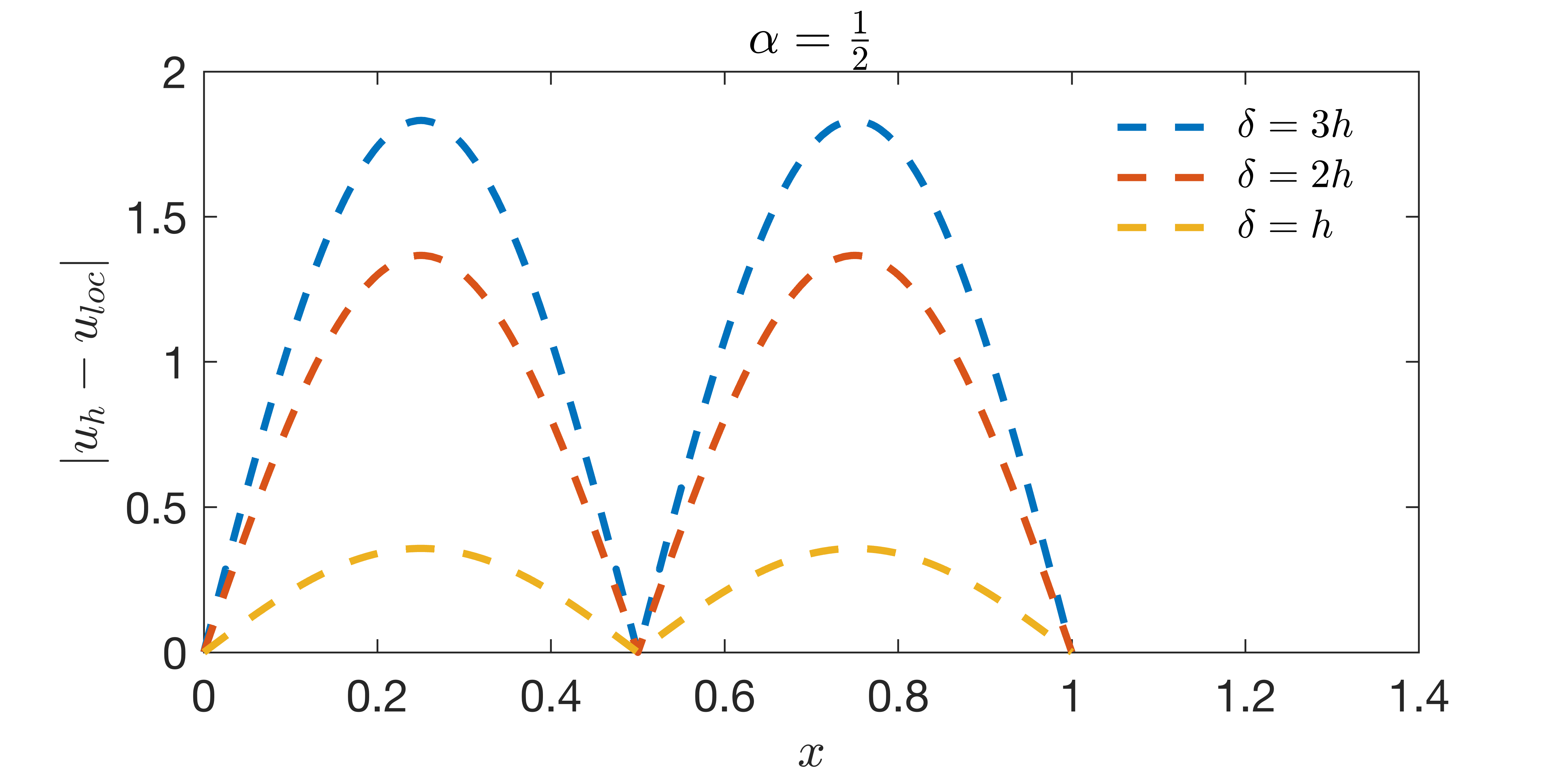} 
\quad \includegraphics[width = 0.45\textwidth]{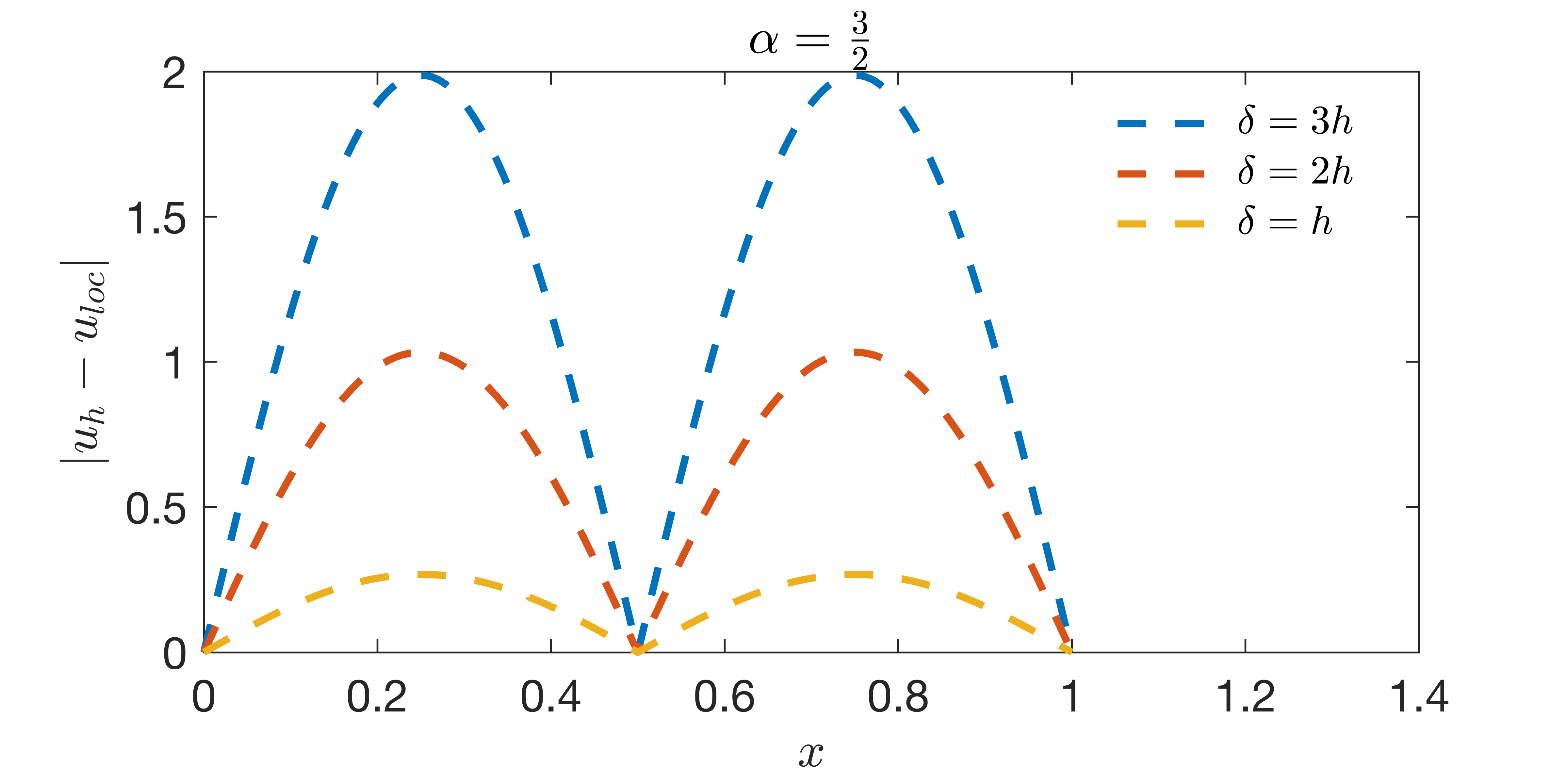}\\
\includegraphics[width = 0.45\textwidth]{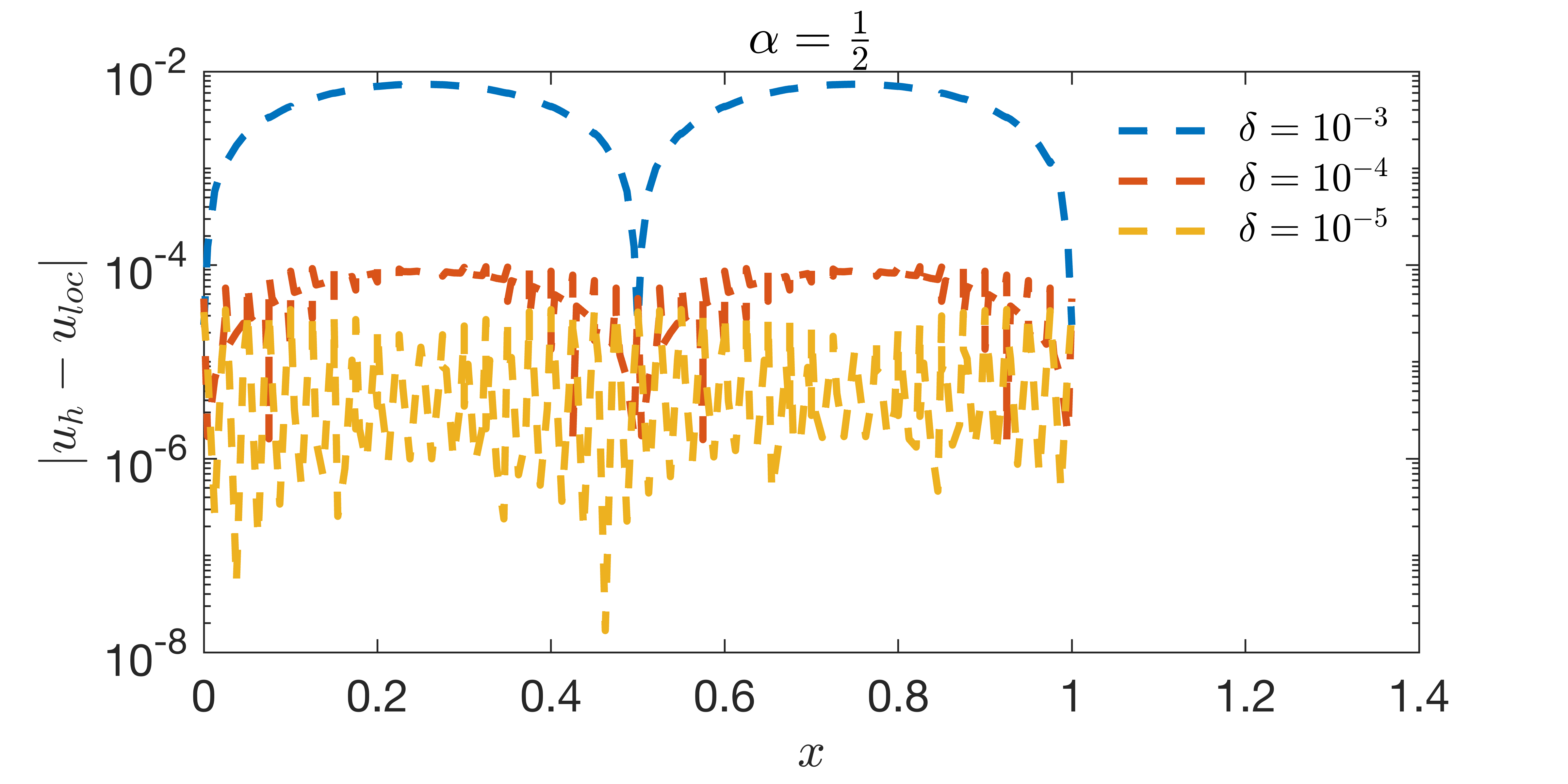}\quad 
\includegraphics[width = 0.45\textwidth]{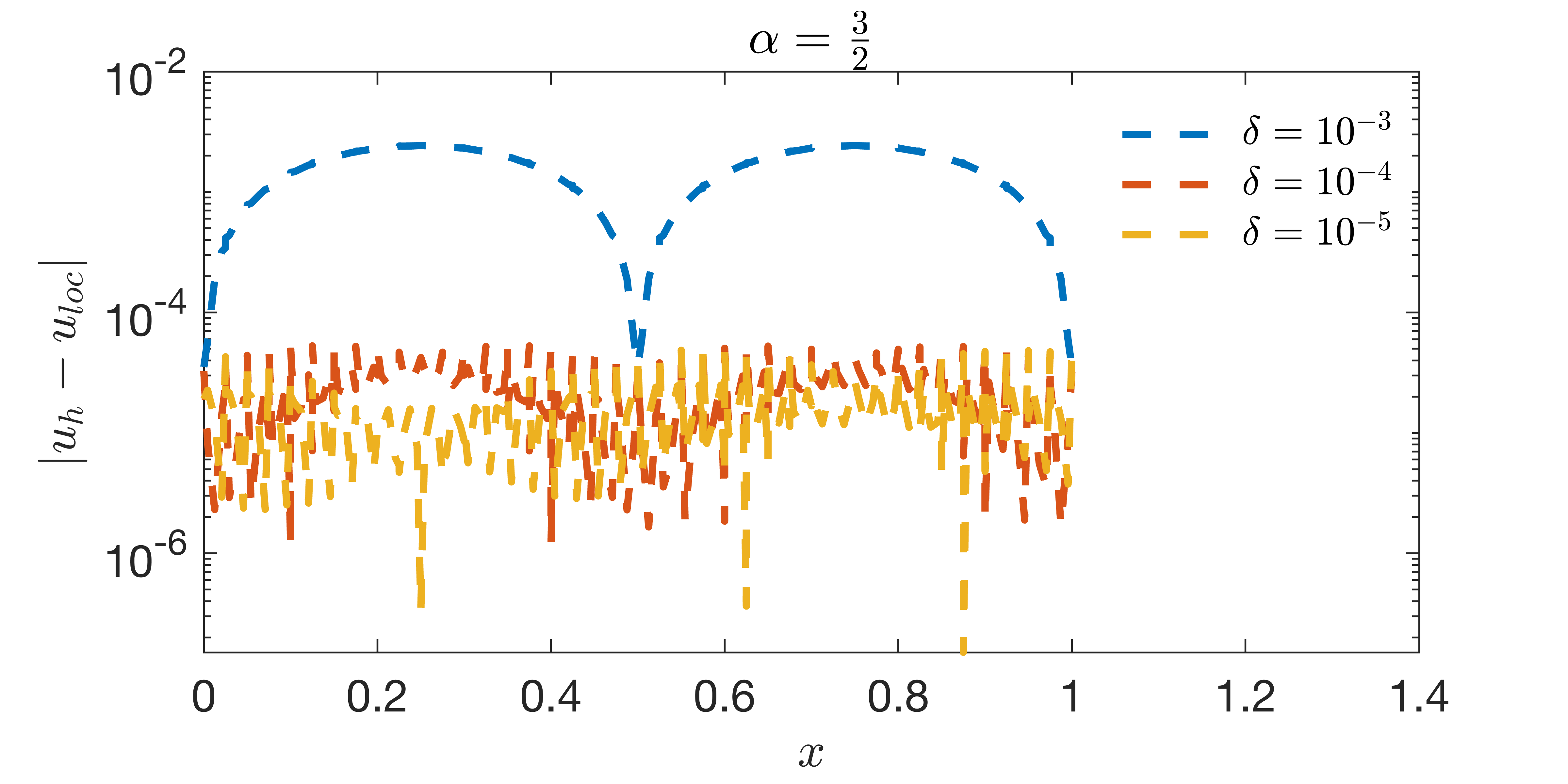}
\caption{\scriptsize{The difference between the numerical solution of the nonlocal problem \eqref{EQ:NW} and the local wave equation \eqref{EQ:wave}. For the spatial discretization we set $h=0.025$ and for the time discretization we use $h_t=0.01$, up to a final time of $T=100$ with the degree of the aproximation space being $k=2$. \textbf{left:} Results for the problem with an integrable kernel of $\alpha = \frac{1}{2}$. \textbf{right:} Results for the problem with a kernel.  }} \label{fig:asy_loc}
\end{figure}

Finally, it is also worth noting that the convergence of the nonlocal wave equation with the nonlocal operator $\cL_\delta$ to its local counterpart is theoretically expected to have an optimal order of $\cO(\delta^2)$ in the maximum norm (see, e.g., \cite{YouLuTaskYu2020}). Indeed, let $u$ and $u_{\text{loc}}$ denote the solutions of the nonlocal problem \eqref{EQ:NW} and the local problem \eqref{EQ:wave}, respectively. When $s^2\gamma(s)\in L_{\text{loc}}^1(\mathbb{R})$, through a straightforward calculation one can show that
\[
u = \Big(\cI - \frac{1}{12}\int_0^\delta s^4\gamma(s) ds\big(\partial_{tt} - \partial_{xx}\big)^{-1} \partial_{xxxx} + \cO(\delta^4)\Big)\big(\partial_{tt} - \partial_{xx}\big)^{-1} \big(u_0'' + \cO(\delta^2)\big) = u_{\text{loc}} + \cO(\delta^2),
\]
where $u_0 = u(x,0)$ denotes the initial data, $''$ represents the second-order derivative with respect to the spatial variable, and $\mathcal{I}$ is the identity operator. Table \ref{tab:convergence_delta} shows the $L^\infty$ norm of the difference between the numerical solution of \eqref{EQ:NW} and the solution of the local wave equation \eqref{EQ:wave} for different horizon values $\delta$. The numerical settings for the spatial and temporal discretization are identical to those used in Figure \ref{fig:asy_loc}. The results show a second order convergence rate for both integrable and non-integrable kernels which agrees with the theoretical findings.

\begin{table}[!htb]
\begin{center}
    \scalebox{1}{
    \begin{tabular}{c c c c c}
                    \hline
                    & $\alpha = \frac{1}{2}$ &  & $\alpha = \frac{3}{2}$ & \\
\hline
 				\hline
 			 $\delta (\times 10^{-2})$  & $\| u_h - u_{loc} \|_{L^\infty}$ & order & $\| u_h - u_{loc} \|_{L^\infty}$ & order \\
 				  \cline{1-5}
               $2^{-1} $ & 1.4586e-01 & -- & 5.8714e-02 & -- \\
            $2^{-2}$ &  4.2314e-02 & 1.7854 & 1.4877e-02 & 1.9806 \\
                  $2^{-3}$ & 1.1378e-02 & 1.8949 & 3.7614e-03 & 1.9837 \\
               $2^{-4}$ & 2.9778e-03 & 1.9339 & 9.7053e-04 & 1.9544 \\
            $2^{-5}$ & 7.8271e-04 & 1.9277 & 2.5810e-05 & 1.9109 \\
 				\hline
    \end{tabular}
 		}
\end{center}
\caption{\scriptsize{$L^\infty$ errors and the corresponding convergence rates of the nonlocal solution $u_h$ to the local solution $u_{\text{loc}}$. The numerical settings are identical to those used in Figure \ref{fig:asy_loc}.}}\label{tab:convergence_delta}
\end{table}

\section{Concluding remarks}
\label{sec:conclusion}

To summarize, we have developed and analyzed a DG method for solving a nonlocal wave equation in one-dimensional space. By introducing an auxiliary variable, analogous to the gradient field in local DG methods for classical PDEs, we have reformulated the problem into a system. The proposed scheme uses a DG method for spatial discretization and the Crank-Nicolson method for time integration. We established the stability of the scheme and derived optimal $L^2$-error estimates for integrable kernels. In addition, we proved that the DG scheme recovers the local DG method in the limit as the horizon $\delta \rightarrow 0$. Furthermore, we show that the fully discrete scheme, when combined with the Crank-Nicolson time integrator, preserves energy conservation. Numerical experiments confirm our theoretical results.

This work also opens up several directions for future research. A key extension is to develop a more comprehensive error analysis for the fully discrete scheme. Another important direction is to extend the proposed method to nonlinear problems that capture physical phenomena beyond the linear setting. In addition, the study of more general nonlocal initial and boundary value problems, including nonlocal Dirichlet and Neumann-type models, would further enhance the applicability of the method. Furthermore, the study of nonlinear nonlocal dynamics and coupled local-nonlocal systems could provide valuable insights into the modeling of complex physical phenomena. We leave these challenges for future studies.

\section*{Acknowledgments}

This work is partially supported by the National Science Foundation through grants DMS-1913309, DMS-1937254, DMS-2309802, and DMS-2513924.  

{\small
\bibliography{BIB-REN}

\begin{thebibliography}{10}

\bibitem{Aksoylu-JPNM23}
{\sc B.~Aksoylu}, {\em On four mutual properties of classical and nonlocal wave equations}, Journal of Peridynamics and Nonlocal Modeling, 5 (2023), pp.~60--80.

\bibitem{AksoyluMengesha2010}
{\sc B.~Aksoylu and T.~Mengesha}, {\em Results on nonlocal boundary value problems}, Numerical functional analysis and optimization, 31 (2010), pp.~1301--1317.

\bibitem{alali2020fourier}
{\sc B.~Alali and N.~Albin}, {\em Fourier spectral methods for nonlocal models}, Journal of Peridynamics and Nonlocal Modeling, 2 (2020), pp.~317--335.

\bibitem{alali2021fourier}
\leavevmode\vrule height 2pt depth -1.6pt width 23pt, {\em Fourier multipliers for nonlocal laplace operators}, Applicable Analysis, 100 (2021), pp.~2526--2546.

\bibitem{bates1999integrodifferential}
{\sc P.~W. Bates and A.~Chmaj}, {\em An integrodifferential model for phase transitions: stationary solutions in higher space dimensions}, Journal of statistical physics, 95 (1999), pp.~1119--1139.

\bibitem{beyer2016class}
{\sc H.~R. Beyer, B.~Aksoylu, and F.~Celiker}, {\em On a class of nonlocal wave equations from applications}, Journal of Mathematical Physics, 57 (2016).

\bibitem{bobaru2010peridynamic}
{\sc F.~Bobaru and M.~Duangpanya}, {\em The peridynamic formulation for transient heat conduction}, International Journal of Heat and Mass Transfer, 53 (2010), pp.~4047--4059.

\bibitem{buades2010image}
{\sc A.~Buades, B.~Coll, and J.-M. Morel}, {\em Image denoising methods. a new nonlocal principle}, SIAM review, 52 (2010), pp.~113--147.

\bibitem{ChGo-ESAIM18}
{\sc F.~A. Chiarello and P.~Goatin}, {\em Global entropy weak solutions for general non-local traffic flow models with anisotropic kernel}, ESAIM: Mathematical Modelling and Numerical Analysis, 52 (2018), pp.~163--180.

\bibitem{ChouShuXing2014}
{\sc C.-S. Chou, C.-W. Shu, and Y.~Xing}, {\em Optimal energy conserving local discontinuous {Galerkin} methods for second-order wave equation in heterogeneous media}, Journal of Computational Physics, 272 (2014), pp.~88--107.

\bibitem{ciarlet2002finite}
{\sc P.~G. Ciarlet}, {\em The finite element method for elliptic problems}, SIAM, 2002.

\bibitem{cockburn2004discontinuous}
{\sc B.~Cockburn}, {\em Discontinuous {Galerkin} methods for computational fluid dynamics}, Encyclopedia of Computational Mechanics,  (2004).

\bibitem{CockburnGopalakrishnanLazarov2009}
{\sc B.~Cockburn, J.~Gopalakrishnan, and R.~Lazarov}, {\em Unified hybridization of discontinuous {Galerkin}, mixed, and continuous {Galerkin} methods for second order elliptic problems}, SIAM Journal on Numerical Analysis, 47 (2009), pp.~1319--1365.

\bibitem{coclite2020numerical}
{\sc G.~M. Coclite, A.~Fanizzi, L.~Lopez, F.~Maddalena, and S.~F. Pellegrino}, {\em Numerical methods for the nonlocal wave equation of the peridynamics}, Applied Numerical Mathematics, 155 (2020), pp.~119--139.

\bibitem{dang2024regularity}
{\sc T.~Dang, B.~Alali, and N.~Albin}, {\em Regularity of solutions for the nonlocal wave equation on periodic distributions}, arXiv:2408.00912,  (2024).

\bibitem{DayalBhattacharya2007}
{\sc K.~Dayal and K.~Bhattacharya}, {\em A real-space non-local phase-field model of ferroelectric domain patterns in complex geometries}, Acta Materialia, 55 (2007), p.~1907–1917.

\bibitem{du2012analysis}
{\sc Q.~Du, M.~Gunzburger, R.~B. Lehoucq, and K.~Zhou}, {\em Analysis and approximation of nonlocal diffusion problems with volume constraints}, SIAM review, 54 (2012), pp.~667--696.

\bibitem{DuHaZhZh-SIAM18}
{\sc Q.~Du, H.~Han, J.~Zhang, and C.~Zheng}, {\em Numerical solution of a two-dimensional nonlocal wave equation on unbounded domains}, SIAM Journal on Scientific Computing, 40 (2018), pp.~A1430--A1445.

\bibitem{du2018numerical}
\leavevmode\vrule height 2pt depth -1.6pt width 23pt, {\em Numerical solution of a two-dimensional nonlocal wave equation on unbounded domains}, SIAM Journal on Scientific Computing, 40 (2018), pp.~A1430--A1445.

\bibitem{DuJuLu-MC19}
{\sc Q.~Du, L.~Ju, and J.~Lu}, {\em A discontinuous {Galerkin} method for one-dimensional time-dependent nonlocal diffusion problems}, Mathematics of Computation, 88 (2019), pp.~123--147.

\bibitem{DuJuLuTian-CAMC20}
{\sc Q.~Du, L.~Ju, J.~Lu, and X.~Tian}, {\em A discontinuous {G}alerkin method with penalty for one-dimensional nonlocal diffusion problems}, Communications on Applied Mathematics and Computation,  (2020), pp.~31--55.

\bibitem{DuJuLuTian-ESAIM24}
\leavevmode\vrule height 2pt depth -1.6pt width 23pt, {\em Numerical analysis of a class of penalty discontinuous {G}alerkin methods for nonlocal diffusion problems}, ESAIM: Mathematical Modelling and Numerical Analysis, 58 (2024), pp.~2035--2059.

\bibitem{du2019asymptotically}
{\sc Q.~Du, Y.~Tao, X.~Tian, and J.~Yang}, {\em Asymptotically compatible discretization of multidimensional nonlocal diffusion models and approximation of nonlocal {Green’s} functions}, IMA Journal of numerical analysis, 39 (2019), pp.~607--625.

\bibitem{DuZh-JMMS23}
{\sc Y.~Du and J.~Zhang}, {\em On perfectly matched layers of nonlocal wave equations in unbounded multiscale media}, Journal of Mechanics of Materials and Structures, 17 (2023), pp.~343--364.

\bibitem{du2019convergence}
{\sc Y.~Du, L.~Zhang, and Z.~Zhang}, {\em Convergence analysis of a discontinuous {G}alerkin method for wave equations in second-order form}, SIAM J. Numer. Anal., 57 (2019), pp.~238--265.

\bibitem{d2020numerical}
{\sc M.~D’Elia, Q.~Du, C.~Glusa, M.~Gunzburger, X.~Tian, and Z.~Zhou}, {\em Numerical methods for nonlocal and fractional models}, Acta Numerica, 29 (2020), pp.~1--124.

\bibitem{EmWe-MMS07}
{\sc E.~Emmrich and O.~Weckner}, {\em Analysis and numerical approximation of an integro-differential equation modeling non-local effects in linear elasticity}, Mathematics and mechanics of solids, 12 (2007), pp.~363--384.

\bibitem{EmWe-CMS07}
\leavevmode\vrule height 2pt depth -1.6pt width 23pt, {\em On the well-posedness of the linear peridynamic model and its convergence towards the {Navier} equation of linear elasticity}, Commun. Math. Sci., 5 (2007), pp.~851--864.

\bibitem{fife2003some}
{\sc P.~Fife}, {\em Some nonclassical trends in parabolic and parabolic-like evolutions}, Trends in nonlinear analysis,  (2003), pp.~153--191.

\bibitem{foss2016differentiability}
{\sc M.~Foss and P.~Radu}, {\em Differentiability and integrability properties for solutions to nonlocal equations}, in New Trends in Differential Equations, Control Theory and Optimization: Proceedings of the 8th Congress of Romanian Mathematicians, World Scientific, 2016, pp.~105--119.

\bibitem{gilboa2009nonlocal}
{\sc G.~Gilboa and S.~Osher}, {\em Nonlocal operators with applications to image processing}, Multiscale Modeling \& Simulation, 7 (2009), pp.~1005--1028.

\bibitem{GroteSchneebeliSchotzau2006}
{\sc M.~J. Grote, A.~Schneebeli, and D.~Sch\"otzau}, {\em Discontinuous {Galerkin} finite element method for the wave equation}, SIAM Journal on Numerical Analysis, 44 (2006), pp.~2408--2431.

\bibitem{guan2015stability}
{\sc Q.~Guan and M.~Gunzburger}, {\em Stability and accuracy of time-stepping schemes and dispersion relations for a nonlocal wave equation}, Numerical Methods for Partial Differential Equations, 31 (2015), pp.~500--516.

\bibitem{HeShSeCySi-IJNME23}
{\sc A.~Hermann, A.~Shojaei, P.~Seleson, C.~J. Cyron, and S.~A. Silling}, {\em Dirichlet-type absorbing boundary conditions for peridynamic scalar waves in two-dimensional viscous media}, International Journal for Numerical Methods in Engineering,  (2023).

\bibitem{nguyen2011high}
{\sc N.~C. Nguyen, J.~Peraire, and B.~Cockburn}, {\em High-order implicit hybridizable discontinuous {Galerkin} methods for acoustics and elastodynamics}, Journal of Computational Physics, 230 (2011), pp.~3695--3718.

\bibitem{ReHi-LASL73}
{\sc W.~Reed and T.~Hill}, {\em Triangular mesh methods for the neutron transport equation,}, Los Alamos Scientific Laboratory report LA-UR-73-479,  (1973).

\bibitem{RenLuZhou2024}
{\sc K.~Ren, L.~Zhang, and Y.~Zhou}, {\em An energy-based discontinuous {Galerkin} method for the nonlinear {Schr\"{o}dinger} equation with wave operator}, SIAM Journal on Numerical Analysis, 62 (2024), pp.~2459--2483.

\bibitem{rosasco2010learning}
{\sc L.~Rosasco, M.~Belkin, and E.~De~Vito}, {\em On learning with integral operators.}, Journal of Machine Learning Research, 11 (2010).

\bibitem{SachsSchu2013}
{\sc E.~W. Sachs and M.~Schu}, {\em A priori error estimates for reduced order models in finance}, ESAIM: Mathematical Modelling and Numerical Analysis, 47 (2013), pp.~449--469.

\bibitem{Silling2000}
{\sc S.~A. Silling}, {\em Reformulation of elasticity theory for discontinuities and long-range forces}, Journal of the Mechanics and Physics of Solids, 48 (2000), pp.~175--209.

\bibitem{tian2017conservative}
{\sc H.~Tian, L.~Ju, and Q.~Du}, {\em A conservative nonlocal convection--diffusion model and asymptotically compatible finite difference discretization}, Computer Methods in Applied Mechanics and Engineering, 320 (2017), pp.~46--67.

\bibitem{tian2013analysis}
{\sc X.~Tian and Q.~Du}, {\em Analysis and comparison of different approximations to nonlocal diffusion and linear peridynamic equations}, SIAM Journal on Numerical Analysis, 51 (2013), pp.~3458--3482.

\bibitem{TiDu-SIAM14}
\leavevmode\vrule height 2pt depth -1.6pt width 23pt, {\em Asymptotically compatible schemes and applications to robust discretization of nonlocal models}, SIAM J. Numer. Anal., 52 (2014), pp.~1641--1665.

\bibitem{tian2015nonconforming}
\leavevmode\vrule height 2pt depth -1.6pt width 23pt, {\em Nonconforming discontinuous galerkin methods for nonlocal variational problems}, SIAM Journal on Numerical Analysis, 53 (2015), pp.~762--781.

\bibitem{vishwanathan2010graph}
{\sc S.~V.~N. Vishwanathan, N.~N. Schraudolph, R.~Kondor, and K.~M. Borgwardt}, {\em Graph kernels}, The Journal of Machine Learning Research, 11 (2010), pp.~1201--1242.

\bibitem{WaYaZh-arXiv22}
{\sc J.~Wang, J.~Z. Yang, and J.~Zhang}, {\em Stability and convergence analysis of high-order numerical schemes with {DtN}-type absorbing boundary conditions for nonlocal wave equations}, arXiv:2211.04307,  (2022).

\bibitem{wangperi}
{\sc L.~Wang, J.~Xu, and J.~Wang}, {\em A peridynamic framework and simulation of non-{Fourier} and nonlocal heat conduction}, International Journal of Heat and Mass Transfer, 118 (2018), pp.~1284--1292.

\bibitem{weckner2005effect}
{\sc O.~Weckner and R.~Abeyaratne}, {\em The effect of long-range forces on the dynamics of a bar}, Journal of the Mechanics and Physics of Solids, 53 (2005), pp.~705--728.

\bibitem{YouLuTaskYu2020}
{\sc H.~You, X.~Lu, N.~Task, and Y.~Yu}, {\em An asymptotically compatible approach for {Neumann}-type boundary condition on nonlocal problems}, ESAIM: Mathematical Modelling and Numerical Analysis, 54 (2020), pp.~1373--1413.

\bibitem{ZhHuDuZh-SIAM17}
{\sc C.~Zheng, J.~Hu, Q.~Du, and J.~Zhang}, {\em Numerical solution of the nonlocal diffusion equation on the real line}, SIAM Journal on Scientific Computing, 39 (2017), pp.~A1951--A1968.

\end{thebibliography}
\bibliographystyle{siam}
}


\end{document}